%% file: DiscreteFundGrp_WarpedConesArxiv3.tex
\documentclass[a4paper,11pt]{amsart}

\usepackage[T1]{fontenc}	

\input{StdCommands_Oct18}

\newcommand{\cone}{\CO} 

\title{Discrete fundamental groups of Warped Cones and expanders}
\author{Federico Vigolo}

\begin{document}

\title{Discrete fundamental groups of warped cones and expanders}

\begin{abstract}
In this paper we compute the discrete fundamental groups of warped cones. As an immediate consequence, this allows us to show that there exist coarsely simply\=/connected expanders and superexpanders. This also provides a strong coarse invariant of warped cones and implies that many warped cones cannot be coarsely equivalent to any box space.
\keywords{Warped cone \and discrete fundamental group \and action \and box space \and coarse equivalence \and presentation \and expanders}
\end{abstract}

\maketitle

\section{Introduction}\label{sec:intro}

The main object of study of this paper is the coarse geometry of warped cones. A warped cone is an infinite metric space with bounded geometry that is constructed starting from an action of a group $\Gamma$ on a base metric space $X$ (see Appendix~\ref{sec:appendix}) for a precise definition). This construction was first introduced by John Roe in \cite{Roe05} as a way of providing examples of metric spaces with prescribed coarse geometric properties.

While passing relatively unnoticed in the years following their introduction, warped cones have lately attracted a great deal of interest and lively research; starting with \cite{DrNo15}, many works followed in quick succession: \cite{Saw18,NoSa17,Vig18b,Saw17,SaWu17,WaWa17,dLVi18,Saw18b} and finally \cite{FNvL17}.  The main interest of warped cones stems from the fact that their coarse geometry encodes information about both the metric space $X$ and the action of the group $\Gamma$ (to the extent that, in some cases, properties of the action are equivalent to geometric properties of the associated warped cones, see \emph{e.g.} \cite{Vig18b,Saw18b,FNvL17}).

In particular, a careful choice of the action $\Gamma\curvearrowright X$ can produce warped cones with very interesting properties. This has successfully been used to provide new examples and counter\=/examples of spaces with bounded geometry that have Yu's Property (A), coarsely embed into Hilbert spaces, coarsely embed into uniformly convex Banach spaces, or satisfy some version of the Baum\textendash Connes conjecture (see \cite{Saw18,NoSa17,SaWu17,WaWa17,Saw18b}).

Another application of warped cones that is similar to the above in spirit but is worth mentioning on its own, is that if the action is expanding then the warped cone represents a family of expander graphs \cite{Vig18b}\textemdash see below for further discussion on this. In particular these techniques can be used to construct new examples of expanders and, the warped cone being a purely geometric object, the families of expanders thus obtained will tend to have a geometry which is relatively easy to describe. Moreover, such expanders will tend to share the properties of the warped cone, so that one can find `special' expanders by looking for `special' warped cones. This train of ideas is at the base of \cite{Saw17,dLVi18}.

\

An \emph{expander} is a sequence of finite graphs that have vertex sets of increasing cardinality, are sparse (\emph{i.e.} have uniformly bounded degrees), but are at the same time highly connected (\emph{i.e.} their Cheeger constant is bounded from below). Such graphs play a very important role in computer science and in the study of computational complexity and, besides their intrinsic interest, they also have important applications in geometry and topology \emph{e.g.} as means of constructing groups and spaces that do not satisfy the Baum\textendash Connes conjecture \cite{HLS02}. See the surveys \cite{HLW06} and \cite{Lub12} for a great introduction to expander graphs and their applications.

The existence of expanders was first proved by probabilistic means \cite{Pin73}, and the first explicit constructions were given only later by Margulis \cite{Mar73}. Since then, a number of ways have been devised for constructing families of expanders. The techniques used for this purpose mostly involve either representation theory, additive combinatorics or zig\=/zag products. 
As already remarked in \cite{Vig18b}, the ideas going into the construction of expanders using warped cones have much in common with older representation\=/theoretical arguments (and have some flavour of additive combinatorics as well), but this construction can provide much wider classes of examples. For example, this is used in \cite{FNvL17} to produce the first explicit examples of a continuum of quasi\=/isometrically distinct (disjoint) expanders and superexpanders (\emph{i.e.} sequences of expanders that do not admit coarsely equivalent subsequences). 

Lately \emph{superexpanders} have been defined as families of expander graphs satisfying some strong spectral/non\=/embeddability property. This stronger requirement makes them much harder to build, the main constructions being due to Lafforgue \cite{Laf08} and Mendel\textendash Naor \cite{MeNa14}. Still, the formalism of warped cones is well\=/suited for carrying over spectral information of this sort, so that it was promptly applied in \cite{Saw17,dLVi18} and later in \cite{FNvL17} to produce new examples of superexpanders (see also \cite{NoSa17}).

We will now enter into more detail on the topics that concern this paper.

\subsection{Warped cones and expanders}
We restrict for now to the study of actions on compact manifolds. Let $(M,g)$ be a compact Riemannian manifold and let $\Gamma=\angles{S}$ be a finitely generated group acting on $M$ by diffeomorphisms. Equip $M\times[1,\infty)$ with the Riemannian metric $d_{\rm cone}=t\cdot g+dt^2$. The \emph{warped cone} as introduced by J.~Roe \cite{Roe05} is the space $\cone_\Gamma(M)=M\times[1,\infty)$ equipped with the warped metric $\wdist$ defined as the maximal metric such that $\wdist\leq d_{\rm cone}$ and $\wdist\bigparen{(x,t),(s\cdot x,t)}\leq 1$ for every $x\in M, t\in [1,\infty)$ and $s\in S$. Note that the warped metric depends on the choice of the finite generating set $S$. Still, different finite generating sets produce \emph{coarsely equivalent} metrics (see Section~\ref{sec:preliminaries} for our conventions and definition of coarse equivalence).

A \emph{level set} of a warped cone is a subset $\cone_\Gamma^t(M)\coloneqq M\times \{t\}\subset\cone_\Gamma(M)$ with the induced metric. In \cite{Vig18b} it is shown how to discretise the level set $\cone_\Gamma^t(M)$ to obtain a finite graph $\CG_t\bigparen{\Gamma\curvearrowright M}$ and it is also proved that, as $t_k\to\infty$, the sequence of graphs $\CG_{t_k}\bigparen{\Gamma\curvearrowright M}$ is a family of expanders if and only if the action $\Gamma\curvearrowright M$ is expanding in measure (this is always the case if the action is measure\=/preserving and it has a spectral gap). This was later improved in \cite{Saw17,dLVi18} where it is shown that such graphs are actually superexpanders as soon as the action has a strong Banach\=/valued spectral gap. 
As explained in \cite{dLVi18}, using the strong Banach property (T) of Lafforgue \cite{Laf08}, it is possible to produce explicit examples of such actions. For instance, the action of the group $\Gamma_d\coloneqq \SO(d,\ZZ[\frac{1}{5}])$ on $\SO(d,\RR)$ by left translation with $d\geq 5$ is such an example (this action was already considered by Margulis in his solution to the Hurwitz problem \cite{Mar80}) and hence the graphs $\CG_{t_k}\bigparen{\Gamma_d\curvearrowright \SO(d,\RR)}$ are superexpanders. Further, in that paper it is also shown that the expanders thus obtained are pairwise \emph{not} coarsely equivalent as $d$ varies (they are in fact quasi\=/isometrically distinct).

\begin{rmk}\label{rmk:intro.wc.wsys}
 A word of warning: here and in the rest of the introduction we always talk about warped cones to seamlessly join the pre\=/existing literature. Still, in the setting of this paper it will be much better to work with what we call \emph{warped systems}: this is the data of the collection of level sets of a warped cone (seen as independent metric spaces) together with the generating set of homeomorphisms $S$. 
 
 The reason why we introduce warped systems is that we are mostly interested in the coarse geometry of the level sets (especially when we use them to construct expanders), and the relation between the coarse geoemtry of a warped cone and that of its level sets is not clear yet. That is, it is not clear when a quasi\=/isometry between warped cones induces uniform quasi\=/isometries between their level sets and, vice versa, it is also not clear whether uniform quasi\=/isometries between level sets of two warped cones can be combined to produce a quasi\=/isometry between the warped cones.
 
 For completeness, we will show in the appendix how to adapt the results of this paper to the settings of warped cones.
\end{rmk}

\subsection{Box spaces and expanders}
The oldest construction of expanders is due to Margulis \cite{Mar73} and is obtained \emph{via} box spaces. A \emph{box space} is a collection of Cayley graphs of finite quotients of a finitely generated group $\Lambda$. Box spaces of $\Lambda$ are denoted by $\Box_{(\Lambda_k)}\Lambda$, where $(\Lambda_k)_{k\in\NN}$ is the collection finite index normal subgroups of $\Lambda$ that we are quotienting by (see Section~\ref{sec:preliminaries} for a more detailed discussion). In particular if $\Lambda$ has property (T) then its box spaces are expanders. If $\Lambda$ has Lafforgue's strong Banach property (T) then its box spaces are superexpanders \cite{Laf08} (the latter are known as \emph{Lafforgue expanders}).

The coarse geometry of box spaces received a fair amount of attention in the past, especially with a view to the Baum\textendash Connes conjectures and in relation to their (non-)embeddability properties into Cayley graphs of groups, Hilbert spaces and Banach spaces. 
In particular various results were obtained trying to distinguish different box spaces up to coarse equivalence.

For example, it was proved in \cite{KhVa17} that if two finitely generated groups $\Gamma$ and $\Lambda$ have coarsely equivalent box spaces then $\Gamma$ and $\Lambda$ must be quasi\=/isometric (see \cite{Das15} for an improvement on this). In particular, this can immediately be used to produce countably many expanders (and superexpanders) that are not coarsely equivalent.

An even more striking result was then proved in \cite{DeKh18}, where discrete fundamental groups (see later) are used to show that if tow finitely presented groups $\Gamma$ and $\Lambda$ have coarsely equivalent box spaces then they are commensurable.

\subsection{Warped cones and box spaces}
Already Roe noted that the coarse geometries of box spaces and of warped cones tend to have very similar behaviour. This is probably due to the fact that for every fixed $r>0$, a ball of radius $r$ in a level set $\cone_\Gamma^t(M)$ with $t\gg r$ looks like a fattening of a ball in the Cayley graph ${\rm Cay}(\Gamma,S)$, and similarly, if $N\lhd\Gamma$ is small enough the ball of radius $r$ in ${\rm Cay}(\Gamma/N,S)$ will be exactly isometric to a ball in ${\rm Cay}(\Gamma,S)$. In this sense, warped cones and box spaces have very similar `local coarse geometries'.

As such, it is natural to try to adapt techniques developed for the study of box spaces to the setting of warped cones. The strategy used in \cite{KhVa17} has a very local flavour to it and it is hence relatively simple to translate in the language of warped cones. This was done in \cite{dLVi18} and it is used to prove a `stable rigidity' result. Specifically, if two essentially free isometric actions on manifolds $\Gamma\curvearrowright M$ and $\Lambda\curvearrowright N$ yield coarsely equivalent warped cones then $\Gamma\times\ZZ^{{\rm dim}(M)}$ and $\Lambda\times\ZZ^{{\rm dim}(N)}$ must be quasi\=/isometric. A similar result was independently obtained in \cite{Saw18b}.

This paper adapts the techniques of \cite{DeKh18} to study warped cones. Still, the results that we obtain are very different from those holding in the case of box spaces. This is because the discrete fundamental group concerns the `global coarse geometry' of metric spaces and it turns out that box spaces and warped cones tend to have very different global geometric properties.

\begin{rmk}
In some sense, warped cones can be thought of as generalisations of box spaces. In fact, it was shown in \cite{Saw18} that every box space is isometric to a warped cone (over a totally disconnected base space). Still, one of the main applications of this paper is to show that in many cases warped cones and box spaces have very different\textemdash even incompatible\textemdash coarse geometry.
\end{rmk}

\subsection{Discrete fundamental groups}
In \cite{BCW14} Barcelo, Capraro and White defined the \emph{discrete fundamental group at scale $\theta$} of a metric space $X$ as the group $\thgrp(X)$ described as the analogue of the fundamental group of $X$ where continuous loops are replaced by closed $\theta$\=/paths (\emph{i.e.} finite sequences of points with $d(x_i,x_{+1})\leq\theta$) which are considered up to $\theta$\=/homotopies\footnote{In the context of simplicial complexes, a similar definition was given in \cite{BKLW01}.} (see Section~\ref{sec:discrete.fund.groups} for a detailed discussion).

From our perspective, the usefulness of the discrete fundamental groups is that the study of the groups $\thgrp(X)$ for (families of) metric spaces can provide some strong coarse invariants. Indeed, even if it is not true in general that $\thgrp(X)$ is invariant under coarse equivalences, it is easy to show that a coarse equivalence $X\to Y$ induces a homomorphism of $\thgrp(X)$ into $\thgrp[\theta'](Y)$ where the parameter $\theta'$ is explicitly bounded in terms of $\theta$ and the constants of the coarse equivalence. 
This information can sometimes be enough to prove that such a coarse equivalence cannot exist.

The main idea of \cite{DeKh18} is to use discrete fundamental groups to study box spaces. What they manage to prove is that if $\Lambda$ is a finitely presented group such that the length of the relations are small with respect to a parameter $\theta$ and $N\lhd\Lambda$ is a small enough normal subgroup, then $\thgrp\bigparen{{\rm Cay}(\Lambda/N)}\cong N$. 
From this result they can deduce that if two box spaces $(\Lambda/\Lambda_k)_{k\in\NN}$ and $(\Gamma/\Gamma_k)_{k\in\NN}$ are coarsely equivalent, then $\Lambda_k\cong\Gamma_k$ for every $k\in\NN$ large enough.

\subsection{Main results}
The idea is to study the discrete fundamental groups of warped cones of actions of the free group $F_S$ generated by the finite set $S$. The case of more general finitely generated groups will be a corollary, as the metric of a warped cone only depends on the set $S$ and hence it does not intrinsically matter whether the acting group is free or not. 
For the sake of simplicity we will limit the current discussion to actions by homeomorphisms on a compact Riemannian manifold $M$, but all the following results are true in much more general settings.

Given a set $S$ of homeomorphisms of $M$, we define the \emph{jumping\=/fundamental group} as the group $\jpgrp(M)$ of closed jumping\=/paths up to homotopy, where a jumping\=/path is a finite sequence of continuous paths whose endpoints differ by an element of $S$. That is, a jumping\=/path is an analogue of a continuous path where we are allowed to `jump' from any point $x\in M$ to a different point $s\cdot x$ using one one of the homeomorphisms $s\in S$. 
It turns out that the jumping\=/fundamental group is isomorphic to a semi\=/direct product $\jpgrp(X)\cong \pi_1(M)\rtimes F_S$. Such an isomorphism can be described explicitly (but it is not natural in general, as it depends on some choices).

Afterwards, we note that given a parameter $\theta\geq 1$, every jumping\=/path can be discretised to obtain a $\theta$\=/path in $M$ (equipped with the warped metric). This yields a (natural) surjection from the jumping\=/fundamental group to the discrete fundamental group at scale $\theta$. The kernel of such surjection can be described explicitly and hence we can completely describe the discrete fundamental group as a quotient of $\pi_1(M)\rtimes F_S$.

At this point, with some care it is possible to find a rather satisfactory description of the discrete fundamental groups of level sets of warped cones. Specifically, we prove the following result (see Theorem~\ref{thm:disc.fund.group.of.warped.cones} for the precise statement).

\begin{thm}\label{thm:intro.discr.fund.grp.wc}
 For every $\theta\geq 1$ there exists a $t_0$ large enough so that for every $t\geq t_0$ we have
 \[
  \thgrp\bigparen{\cone_{F_S}^t(M)}\cong\bigparen{\pi_1(M)\rtimes F_S}\big/\aangles{K_\theta}
 \]
 where $K_\theta$ can be described explicitly and depends on the set of elements $w\in F_S$ for which the homeomorphism $w\colon M\to M$ has fixed points.
\end{thm}

The above result has various immediate consequences. For example, we have the following:

\begin{cor}\label{cor:intro.free.action.simply.connected.mfld}
 If $\Gamma$ is finitely presented, $\Gamma\curvearrowright M$ is a free action and $\pi_1(M)=\{0\}$, then 
 \[  
  \thgrp\bigparen{\cone_\Gamma^t(M)}\cong\Gamma
 \]
 for every $\theta$ and $t$ large enough.
\end{cor}

We also prove more precise statements (see Corollary~\ref{cor:discr.fund.grp.of.free.Gamma.warped.cone.exact.sequence}) that allow us to describe $\thgrp\bigparen{\cone_\Gamma^t(M)}$ for free actions on manifolds that are not simply\=/connected by using a short exact sequence.

\

Since the graphs obtained as discretisations of level sets of a warped cone are coarsely equivalent to the level sets themselves, their discrete fundamental group will (roughly) be the same as that of the level sets. In particular, if $d\geq 3$ is odd and $S\subset\SO(d,\RR)$ is a finite subset so that the action $F_S\curvearrowright \SS^{d-1}$ has a spectral gap, it follows from \cite{Vig18b} that the level sets are expanders. Theorem~\ref{thm:intro.discr.fund.grp.wc} hence implies:

\begin{cor}
 The graphs $\CG_{t_k}\bigparen{F_S\curvearrowright \SS^{d-1}}$ are a family of coarsely simply\=/connected expander graphs. That is,
 \[
  \thgrp\Bigparen{\CG_{t_k}\bigparen{F_S\curvearrowright \SS^{d-1}}}=\{0\}
 \]
 for every $\theta\geq 1$.
\end{cor}

From \cite{dLVi18} we also deduce:

\begin{cor}
 If a finitely generated subgroup $\Gamma=\angles{S}\subset\SO(d,\RR)$ has Lafforgue's strong Banach property (T) and $d$ is odd, then the graphs $\CG_{t_k}\bigparen{F_S\curvearrowright \SS^{d-1}}$ are a family of coarsely simply\=/connected superexpanders.
\end{cor}

\begin{rmk}
 Note that it follows from the work of Delabie\textendash Khukhro \cite{DeKh18} that expanders coming from box spaces cannot be coarsely simply\=/connected (see the discussion in Section~\ref{sec:wc.and.box}).
\end{rmk}

From Theorem~\ref{thm:intro.discr.fund.grp.wc} it follows that, once $\theta$ is fixed, the discrete fundamental group of the level set $\cone_{F_S}^{t}(M)$ does not depend on $t$ (as soon as it is large enough). In many instances, it is also the case that (the limit for large $t$) of the $\theta$\=/group $\thgrp\bigparen{\cone_{F_S}^{t}(M)}$ does not even depend on $\theta\gg 1$. 
When this happens, we say that the warped cone has \emph{stable discrete fundamental group} and we denote such limit by $\thgrp[\infty]\bigparen{F_S\curvearrowright M}$ (in general, we can define this group as the direct limit of the discrete fundamental groups).
As an example, we have that if a finitely generated group $\Gamma$ acts freely on $M$, then the warped cone has stable discrete fundamental group if and only if $\Gamma$ is finitely presented (Lemma~\ref{lem:stable.iff.finitely.presented(free.action)}).

This discrete fundamental group `at scale infinity' can be used as a coarse invariant. In fact we have:

\begin{thm}\label{thm:intro.infinite.discr.grp.is.invariant}
 If (the level sets of) two warped cones $\cone_\Gamma(M)$ and $\cone_\Lambda(N)$ are coarsely equivalent and $\cone_\Gamma(M)$ has stable discrete fundamental group, then $\cone_\Lambda(N)$ has stable discrete fundamental group as well and
 \[
  \thgrp[\infty]\bigparen{\Gamma\curvearrowright M}\cong \thgrp[\infty]\bigparen{\Lambda\curvearrowright N}.
 \]
\end{thm}

Finally, the fact that the discrete fundamental groups of $\cone_{F_S}^{t}(M)$ do not depend on $t$ should be contrasted with the theorem of Delabie\textendash Khukhro \cite{DeKh18} stating that for box spaces of finitely presented groups we have $\thgrp\bigparen{{\rm Cay}(\Lambda/\Lambda_k)}\cong \Lambda_k$ for every $k$ large enough. This provides strong evidence that (level sets of) warped cones and box spaces are generally quasi\=/isometrically distinct. In particular, we can prove the following:

\begin{thm}\label{thm:intro.wc.ce.to.box.spaces}
 If a box space $\Box_{(\Lambda_k)}\Lambda$ is coarsely equivalent to a sequence of level sets of a warped cone $\cone_\Gamma(M)$, then $\cone_\Gamma(M)$ has stable discrete fundamental group if and only if $\Lambda$ is finitely presented. When this is the case, $\Lambda_k\cong \thgrp[\infty]\bigparen{\Gamma\curvearrowright M}$ for every $k\in\NN$ large enough.
\end{thm}

This result allows us to show that for various classes of groups $\Lambda$ (\emph{e.g.} free; with fixed price $p>1$; lattices in simple Lie groups not locally isomorphic to $\Sl(2,\RR)$) no box space of $\Lambda$ can be coarsely equivalent to a warped cone of this sort.

Vice versa, (level sets of) a warped cone $\cone_\Gamma(M)$ with stable discrete fundamental group cannot be coarsely equivalent to any box space if $\thgrp[\infty]\bigparen{\cone_\Gamma(M)}$ is \emph{e.g.} finite; not finitely presented; not residually finite; a lattice in a higher rank Lie group or, more generally, \emph{finitely co\=/Hopfian} (see \cite{vLi17} for a definition and discussion of these groups). Moreover, if the $\Gamma$\=/action is free and $\pi_1(M)$ is finite, then the level sets cannot be coarsely equivalent to a box space if $\Gamma$ has polynomial growth, has property (T) or is hyperbolic.

Using a rigidity result from \cite{dLVi18} (see also \cite{Saw18b}), we manage to extend (a consequence of) \cite[Theorem 5.11]{dLVi18}, proving the following.

\begin{thm}
 The superexpanders $\CG_{t_k}\bigparen{\Gamma_d\curvearrowright\SO(d,\RR)}$ are not coarsely equivalent to any box space.
\end{thm}

Despite all this evidence, in some cases warped cones and box spaces \emph{are} coarsely equivalent. Most notably this happens for actions by rotations on tori and quotients of $\ZZ^d$. There also are examples of warped cones yielding expanders that are coarsely equivalent to box spaces. For such an example (see Example~\ref{exmp:wc.qi.box.expander}) we are indebted to Wouter van Limbeek for the discussions we had at the INI in Cambridge.

\subsection{Historical note}The main results of this paper have been known to the author since
the beginning of 2017, with the core of the paper (including
Theorem~\ref{thm:intro.discr.fund.grp.wc}) already written as of March
the 1$^{\rm st}$, and discussed with a number of participants in the
programme “Non-positive curvature, group actions and cohomology”, held at the Isaac Newton Institute for Mathematical Sciences (Cambridge, UK). The author likewise
explained such results in a seminar talk in Neuch\^{a}tel on the 2$^{\rm
nd}$ of May.

Since then, this work has been in a dormant state, as the author was looking for further consequences and applications. This quiet period was abruptly ended by the sudden and very much unexpected appearance of a preprint of the independent work of Fisher\textendash Nguyen\textendash van Limbeek \cite{FNvL17}.

From what the author could see, discrete fundamental groups of warped cones play a central role in \cite{FNvL17} as well. The results of that paper are focused on the restricted setting of free actions by isometries on manifolds. Moreover, their study uses a somewhat different (and complementary) point of view, as their investigation is more reminiscent of the theory of universal covers (whilst ours is more explicit and hands on). It appears to us that their results do reprove part of our work in the restricted setting they consider (most notably Corollary~\ref{cor:intro.free.action.simply.connected.mfld}), but then they go further and prove a very strong rigidity result that allows them to produce a continuum of non\=/coarsely equivalent warped cones.

\subsection{Organisation of the paper} 
As announced in Remark~\ref{rmk:intro.wc.wsys}, for the development of this paper it is much more convenient to talk about warped systems (Definition \ref{def:warped.system}) rather than warped cones. Such systems are only properly defined in Section~\ref{sec:limit.of.fund.grps} as everything preceding it is a more general discussion.

In Section \ref{sec:preliminaries} we introduce some notation and recall some notions of topology and coarse geometry that will be used throughout. In Section~\ref{sec:discrete.fund.groups}, \ref{sec:jumping.fund.grp} and \ref{sec:discretisations} we properly define the discrete fundamental group and the jumping\=/fundamental group and we prove that discrete fundamental groups can be computed using discretisations of jumping\=/paths. 
Section \ref{sec:criteria.for.computations} only contains some lemmas and criteria to help to compute discrete fundamental groups for warped metrics. With the exception of Proposition~\ref{prop:discr.fund.group.is.semidirect.Gamma}, these results are not necessary for the rest of the paper.

In Section \ref{sec:coarse.fund.grp.wc} we prove Theorem~\ref{thm:intro.discr.fund.grp.wc} and we begin deducing the first corollaries. In Section~\ref{sec:limit.of.fund.grps} we define the coarse fundamental groups $\thgrp[\infty]\bigparen{F_S\curvearrowright M}$, we discuss stable discrete fundamental groups and prove Theorem~\ref{thm:intro.infinite.discr.grp.is.invariant} and its corollaries. Finally, in Section~\ref{sec:wc.and.box} we use these newly\=/developed techniques to study the relations between box spaces and warped systems: we prove Theorem~\ref{thm:intro.wc.ce.to.box.spaces}, we discuss its above mentioned consequences and we also describe the examples of warped systems coarsely equivalent to box spaces.

In the appendix we deal with actual warped cones as defined by Roe and we show that discrete fundamental groups can be similarly used to produce coarse invariants.

\subsection{acknowledgements}
I wish to thank my advisor Cornelia Dru\c{t}u for her continuous support and encouragement and
David Hume, Richard Wade and Gareth Wilkes for helpful discussions and for pointing out the example in Remark~\ref{rmk:non.fin.pres.acting}. I should also thank Wouter van Limbeek for drawing my attention to Example~\ref{exmp:wc.qi.box.expander}. 

The author was funded by the EPSRC Grant 1502483 and the J.T.Hamilton Scholarship. This work was also supported by the NSF under Grant No. DMS-1440140 and the GEAR Network (DMS 1107452, 1107263, 1107367) while the author was in residence at the MSRI in Berkeley during the Fall 2016 semester. We also thank the Isaac Newton Institute for Mathematical Sciences, Cambridge (EPSRC grant no EP/K032208/1), for support and hospitality during the programme “Non-positive curvature, group actions and cohomology”.

\section{Conventions and preliminaries}\label{sec:preliminaries}

\subsection{Topology}
In what follows, $(X,d)$ will always be a metric space. We will sometimes drop the distance function from the notation when no confusion can arise.

The space $X$ is always assumed to be path\=/connected and locally path\=/connected. The extra requirements we need on $X$ will be specified as they vary throughout the paper. They generally include (some version of) either:  rectifiability of paths, geodesicity, compactness and/or semi\=/local simple connectedness.

Let us recall that a continuous path $\gamma\colon[0,1]\to X$ is \emph{rectifiable} if the supremum
\[
 \abs{\gamma}\coloneqq\sup\Bigbrace{\sum_{i=1}^n d\bigparen{\gamma\paren{i-1},\gamma(i)}}
\]
is finite, where the supremum is taken over any finite sequence of times $0=t_0<\cdots<t_n=1$. When this is the case, we say that $\abs{\gamma}$ is the \emph{length} of $\gamma$.

A metric space is \emph{geodesic} if for every couple of points $x,y\in X$ there exists a continuous rectifiable path $\gamma$ connecting them such that $\abs{\gamma}=d(x,y)$. We will need variations on this concept later on (see Definition~\ref{de:theta.geodesic} and~\ref{de:jumping.geodesic}).

The space $X$ is \emph{semi\=/locally simply connected} if every point has a neighbourhood $U\subseteq X$ so that every path contained in $U$ is null\=/homotopic in $X$ (but not necessarily null\=/homotopic in $U$). Recall that this very mild condition is a necessary requirement for the standard proof that a topological space admits a universal cover.

\subsection{Coarse geometry}
Given a constant $C\geq 0$, we say that two maps  $f,g\colon (X,d_X)\to(Y,d_Y)$ between metric spaces are \emph{$C$\=/close} if $d_Y(f(x),g(x))\leq C$ for every $x\in X$.

\begin{de}
 Let $f\colon (X,d_X)\to (Y,d_Y)$ be a (possibly discontinuous) map between metric spaces and let $A,L\geq 0$ be constants. We say that the map $f$ is 
 \begin{itemize}
  \item \emph{$(L,A)$\=/coarsely Lipschitz} if
    \[
    d_Y(f(x),f(x'))\leq Ld_X(x,x')+A   
    \]
    for every $x,x'\in X$;
  \item a \emph{$(L,A)$\=/quasi\=/isometric embedding} if 
    \[
    L^{-1}d_X(x,x')-A\leq d_Y(f(x),f(x'))\leq Ld_X(x,x')+A   
    \] 
    for every $x,x'\in X$;
  \item a \emph{$(L,A)$\=/quasi\=/isometry} if it is $(L,A)$\=/coarsely Lipschitz and there exists an $(L,A)$\=/coarsely Lipschitz quasi\=/inverse $\bar f\colon (Y,d_Y)\to (X,d_X)$ so that $f\circ \bar f$ and $\bar f\circ f$ are $A$\=/close to $\id_Y$ and $\id_X$ respectively.
 \end{itemize}
\end{de}

\begin{rmk}
 This definition of quasi\=/isometry is slightly different from the usual one (coarse surjective quasi\=/isometric embedding). This difference is very much inessential and it only changes the actual values of the constants $A$ and $L$. We choose this definition as it fits more nicely with Lemma~\ref{lem:qi.induce.maps.on.disc.fund.grp}.
\end{rmk}

In general the relation of \emph{coarse equivalence} between metric spaces (see \emph{e.g.} \cite{Roe03}) is weaker than quasi\=/isometry. Still, all the metric spaces considered in this paper will be roughly geodesic (in the sense of Definition~\ref{de:theta.geodesic}) and in this setting coarse equivalences and quasi\=/isometries coincide. With a slight abuse of notation, we will hence continue using the unspoken convention used in the introduction and reserve the term `coarsely equivalent' for families of spaces that are uniformly quasi isometric:

\begin{de}
 We say that two sequences of metric spaces $(X_n,d_{X_n})_{n\in \NN}$ and $(Y_n,d_{Y_n})_{n\in \NN}$ are \emph{coarsely equivalent} if $(X_n,d_{X_n})$ and $(Y_n,d_{Y_n})$ are \emph{uniformly} quasi\=/isometric. That is, if there exist constants $L$ and $A$ such that $(X_n,d_{X_n})$ and $(Y_n,d_{Y_n})$ are $(L,A)$\=/quasi\=/isometric for every $n\in\NN$.
 
 We say that they are \emph{coarsely distinct} if no subsequence of $(X_n,d_{X_n})$ is coarsely equivalent to a subsequence of $(Y_n,d_{Y_n})$.
\end{de}

\subsection{Actions and warped metrics}
In what follows, $S$ will always be a finite set of homeomorphisms of $X$. In particular, the choice of $S$ induces an action by homeomorphisms $F_S\curvearrowright X$ of the free group generated by $S$.

Even when we talk about generic actions of groups $\Gamma\curvearrowright X$ we implicitly assume that a finite generating set $S$ of $\Gamma$ is fixed, and hence this action is just a quotient action of the action of the free group $F_S\curvearrowright X$. 

\begin{rmk}
 Throughout the paper we think of the elements of $S$ as being distinct non\=/trivial homeomorphisms. This is not a requirement, as the formalism that we use still makes sense if $S$ contains duplicates or other trivialities; still we felt useful to point this out as in other works (\emph{e.g.} \cite{Vig18b}) $S$ is used to denote a symmetric generating set of a group $\Gamma$ and is also assumed to contain the identity element.
\end{rmk}

\begin{de}
  Let $S$ be a finite set of homeomorphisms of $(X,d)$. The \emph{warped metric} induced on $X$ by $S$ is the largest metric $\wSdist$ such that 
  \begin{itemize}
   \item $\wSdist(x,y)\leq d(x,y)$ for every $x,y\in X$;
   \item $\wSdist(x,s(x))\leq 1$ for every $s\in S$.
  \end{itemize}
\end{de}

Note that the warped metric $\wSdist$ trivially coincides with the warped metric $\wdist$ (or $d_\Gamma$) as already defined in literature, but makes explicit the dependence on $S$.

From Section~\ref{sec:coarse.fund.grp.wc} on, we will use the following notation: given $t\geq1$, we denote by $\dist[t]=t\cdot \dist$ the distance $\dist$ scaled by $t$ and, given an action $F_S\curvearrowright X$, we let $\wSdist[t]$ denote the warping of the metric $\dist[t]$ induced by the action. The spaces $(X,\wSdist[t])$ are the spaces that we are really going to be interested in and they represent the level sets the warped cone $\cone_S(X)$. See the beginning of Appendix~\ref{sec:appendix} for a discussion on the original definition of warped cones and the relation with their level sets.

\subsection{Box spaces}
Given a residually finite group $\Lambda$, a \emph{residual filtration} is a nested sequence of finite index normal subgroups $\Lambda_{k+1}\lhd\Lambda_k\lhd\Lambda$ such that $\bigcap_{k\in\NN}\Lambda_k=\{e\}$.

\begin{de}
 A \emph{box space} $\Box_{(\Lambda_k)}\Lambda$ of a finitely generated group $\Lambda$ is the family of Cayley graphs of the finite quotients $\bigparen{ {\rm Cay}(\Lambda/\Lambda_k,S)}_{k\in\NN}$, where $(\Lambda_k)_{k\in\NN}$ is a resdual filtration of $\Lambda$.
\end{de}

According to our convention, a sequence $(X_k)_{k\in\NN}$ of metric spaces is coarsely equivalent to a box space $\Box_{(\Lambda_k)}\Lambda$ if and only if $X_k$ is quasi\=/isometric to ${\rm Cay}(\Lambda/\Lambda_k)$ for every $k\in\NN$ and these quasi\=/isometries have uniform constants as $k$ varies.

\begin{rmk}
 In literature, box spaces are often made into a single metric space (as opposed to a family of spaces) by giving the disjoint union of Cayley graphs ${\rm Cay}(\Lambda/\Lambda_k)$ a metric by imposing that the distance between two components be larger than the sum of their diameters. One can hence study the coarse geometry of a box space treating it as a single object.
 
 As long as one is coherent, it does not matter what viewpoint is taken. Indeed, it is shown in \cite{KhVa17} that if two box spaces $\Box_{(\Lambda_k)}\Lambda$ and $\Box_{(\Gamma_k)}\Gamma$ are coarsely equivalent (when seen as metric objects) then, possibly discarding a finite number of indices, one has that ${\rm Cay}(\Lambda/\Lambda_k)$ is uniformly quasi\=/isometric to ${\rm Cay}(\Gamma/\Gamma_k)$.
\end{rmk}

\section{Discrete fundamental group}\label{sec:discrete.fund.groups}
Let $(X,d)$ be a metric space and $\theta>0$ be a fixed a parameter. A \emph{discrete path} at \emph{scale $\theta$} (or \emph{$\theta$\=/path}) is a $\theta$\=/Lipschitz map $Z\colon [n]\to X$ where $[n]\coloneqq \{0,1,2,\ldots,n\}\subset \RR$. Equivalently, $Z$ can be seen as an ordered sequence of points $(z_0,\ldots,z_n)$ in $X$ with $d(z_i,z_{i+1})\leq\theta$; we will use both notations throughout the paper. The space $(X,d)$ is said to be \emph{$\theta$\=/connected} if any two points of $X$ are connected via a $\theta$\=/path. 
\begin{de}\label{de:theta.geodesic}
  A $\theta$\=/connected metric space is \emph{$\theta$\=/geodesic} if any two points $z,z'\in X$ are connected via a $\theta$\=/path $(z_0,\ldots,z_n)$ such that $d(z,z')=d(z_0,z_1)+\cdots+d(z_{n-1},z_n)$.  
\end{de}

A $\theta$\=/path $Z'\colon[m]\to X$ is a \emph{lazy version} (or \emph{lazification}) of $Z$ if it is obtained from it by repeating some values, \emph{i.e.} if $m>n$ and $Z'=Z\circ f$ where $f\colon [m]\to[n]$ is a surjective monotone map. 

Given two $\theta$\=/paths $Z_1 $ and $Z_2$ of the same length $n$, a \emph{free $\theta$\=/grid homotopy} between them is a $\theta$\=/Lipschitz map $\widehat H\colon [n]\times[m]\to X$ such that $\widehat H(\;\!\cdot\;\!,0)=Z_1$ and $\widehat H(\;\!\cdot\;\!,m)=Z_2$. Here the product $[n]\times[m]$ is equipped with the $\ell^1$ metric. A \emph{$\theta$\=/grid homotopy} is a free $\theta$\=/grid homotopy so that $\widehat H(0,t)=Z_1(0)=Z_2(0)$ and $\widehat H(n,t)=Z_1(n)=Z_2(n)$ for every $t\in[m]$. 

\begin{de}
 Two $\theta$\=/paths $Z_1$ and $Z_2$ are \emph{$\theta$\=/homotopic} (denoted by $Z_1\thhom Z_2$) if they are equivalent under the equivalence relation induced by lazifications and $\theta$\=/grid homotopies. Equivalently, $Z_1\thhom Z_2$ if and only if there exist $Z_1'$ and $Z_2'$ lazy versions of $Z_1$ and $Z_2$ which are homotopic via a single $\theta$\=/grid homotopy.
\end{de}

\begin{de}[\cite{BCW14}]
 Let $(X,d)$ be a $\theta$\=/connected metric space and $x_0\in X$ a fixed base point. The \emph{discrete fundamental group at scale $\theta$} is the group $\thgrp\bigparen{(X,d),x_0}$ of $\theta$\=/homotopy classes of closed $\theta$\=/paths with endpoints $x_0$, equipped with the operation of concatenation. When the metric is clear from the context, we will drop it from the notation and simply write $\thgrp(X,x_0)$.
\end{de}

Just as for the usual fundamental group, the isomorphism class of the discrete fundamental group does not depend on the choice of the base point. Moreover, if $x_0$ and $x_0'$ are points at distance at most $\theta$, then the map sending a $\theta$\=/path $(x_0,x_1,\ldots,x_0)$ to the $\theta$\=/path $(x_0',x_0,x_1,\ldots,x_0,x_0')$ induces an isomorphism $I_{(x_0,x_0')}\colon\thgrp(X,x_0)\xrightarrow{\cong}\thgrp(X,x_0')$.

\

If $f\colon (X,d_X)\to (Y,d_Y)$ is a $(L,A)$\=/coarse Lipschitz map and $\theta'\geq L\theta+A$, then $f$ induces a homomorphism $f_*\colon\thgrp(X,x_0)\to\thgrp[\theta'](Y,f(x_0))$ by composition: $f_*([Z])\coloneqq[f\circ Z]$. We collect some basic results about discrete fundamental groups in the following lemma (some of these statements are also proved in \cite{BCW14}). 

\begin{lem}\label{lem:qi.induce.maps.on.disc.fund.grp}
 Let $(X,d_X)$ and $(Y,d_Y)$ be $1$\=/geodesic metric spaces; $x_0\in X$ a fixed base point; $f\colon (X,d_X)\to (Y,d_Y)$ an $(L,A)$\=/coarse Lipschitz map; $\theta,\theta'>1$ constants with $\theta'\geq L\theta+A$ and $f_*\colon\thgrp(X,x_0)\to\thgrp[\theta'](Y,f(x_0))$ the induced homomorphism. Then the following hold.
 \begin{enumerate}[(i)]
  \item\label{item:lem:theta'.close} If $g\colon (X,d_X)\to (Y,d_Y)$ is $(L,A)$\=/coarse Lipschitz and $\theta'$\=/close to $f$, then $g_*=I_{(f(x_0),g(x_0))}\circ f_*$.
  
  \vspace{0.6 ex}
  \item\label{item:lem:surjective} If $f$ is an $(L,A)$\=/quasi\=/isometry, and $\theta\geq L+A$
  then $f_*$ is surjective.
  
  \vspace{0.6 ex}
  \item\label{item:lem:change.parameter} If $\theta'\geq \theta$ then $(\id_X)_*\colon\thgrp(X,x_0)\to\thgrp[\theta'](X,x_0)$ is surjective.
  
  \vspace{0.6 ex}
  \item\label{item:lem:isomorphism} If $f$ is an $(L,A)$\=/quasi isometry, $\theta\geq L+A$ and $(\id_X)_*\colon\thgrp(X,x_0)\to\thgrp[L\theta'+A](X,x_0)$ is an isomorphism, then $f_*$ is an isomorphism.
 \end{enumerate}
\end{lem}

\begin{proof}
 (\ref{item:lem:theta'.close}): Given a closed $\theta$\=/path $Z$ with endpoints $x_0$, the path $I_{(g(x_0),g(x_0))}\circ g(Z)$ is a lazification of $g(Z)$. It is now enough to notice that the paths $I_{(f(x_0),g(x_0))}\circ f(Z)$ and $I_{(g(x_0),g(x_0))}\circ g(Z)$ are $\theta'$\=/grid homotopic because the maps $f$ and $g$ are $\theta'$\=/close.
 
 (\ref{item:lem:surjective}): Since $(Y,d_Y)$ is $1$\=/geodesic, any $\theta'$\=/path $Z'$ in $Y$ is $\theta'$\=/homotopic to a $1$\=/path. Indeed, one can concatenate the $1$\=/paths realising the distance between consecutive points $Z'(i),Z'(i+1)$ obtaining this way a $1$\=/path that is $\theta'$\=/homotopic to $Z'$. 
 
 Given an element $[Z']\in \thgrp[\theta'](Y,f(x_0))$, we can assume that $Z'$ is a $1$\=/path. Let $\bar f$ be a quasi\=/inverse of $f$, then $\bar f(Z')$ is a $(L+A)$\=/path and hence a $\theta$\=/path. It is then easy to see that $I_{(\bar f(f(x_0)),x_0)}(\bar f (Z'))$ is a closed $\theta$\=/path based at $x_0$ whose image under $f$ is $\theta'$\=/homotopic to $Z'$. 
 
 (\ref{item:lem:change.parameter}): This is a special case of (\ref{item:lem:surjective}).
 
 (\ref{item:lem:isomorphism}): The quasi\=/inverse $\bar f$ induces a homomorphism $\bar f_*\colon \thgrp[\theta'](Y,f(x_0))\to\thgrp[L\theta'+A](X,\bar f\circ f(x_0))$. As $\bar f\circ f$ is $A$\=/close to $\id_X$ and therefore $\theta$\=/close, by (\ref{item:lem:theta'.close}) we have that $(\bar f\circ f)_*$ coincides with $I_{(x_0,\bar f\circ f(x_0))}\circ (\id_X)_*$ and it is hence an isomorphism. The claim follows because $f_*$ is surjective by (\ref{item:lem:surjective}).
\end{proof}

\section{Jumping-fundamental group}\label{sec:jumping.fund.grp}

Let $S$ be a finite set and let $F_S$ denote the free group freely generated by $S$. Given an action by homeomorphisms $F_S\curvearrowright (X,d)$, let $\wSdist$ denote the \emph{warped metric} induced on $X$ \emph{i.e.} the maximal metric such that $\wSdist\leq d$ and $\wSdist(x,s(x))\leq 1$ for every $s\in S$. In what follows we will always assume that $X$ is path\=/connected.

\begin{de}\label{de:jumping.geodesic}
 A \emph{\jp} in $X$ is a finite sequence of continuous paths $\gamma_0,\ldots,\gamma_n\colon [0,1]\to X$ and elements $\vec{s}_1,\ldots,\vec s_n$ with $\vec s_i=s_i^\pm$ for some $s_i\in S$, such that $\gamma_{i}(0)=\vec s_i\cdot\gamma_{i-1}(1)$ for every $i=1,\ldots,n$. Such \jp will be denoted by
 \[
  \vec{\bm\gamma} \coloneqq \gamma_0 \jto[\vec s_1] \gamma_1 \jto [\vec s_2]\cdots\jto[\vec s_n]\gamma_n.
 \]
 The \emph{jumping pattern} of $\vec{\bm\gamma}$ is the ordered sequence $(\vec s_1,\ldots,\vec s_n)$.
\end{de}

We define the \emph{length} of the \jp $\vec{\bm\gamma}$ as
\[
 \norm{\vec{\bm\gamma}}\coloneqq n+\sum_{i=0}^n\norm{\gamma_i},
\]
where $\norm{\gamma_i}$ is the length of the path $\gamma_i$ in $(X,d)$. Note that if $X$ is a geodesic space the warped distance can then be expressed as
 \[
  \wSdist(x,y)=\inf\big\{\norm{\vec{\bm\gamma}}\bigmid \vec{\bm\gamma}\text{ \jp between $x$ and $y$}\big\}. 
 \]

\begin{de}
 A \jp $\vec{\bm\gamma}$ between two points $x,y\in X$ is \emph{geodesic} if it realises their distance $\norm{\vec{\bm\gamma}}=\wSdist(x,y)$. The warped metric space $(X,\wSdist)$ is \emph{jumping\=/geodesic} if $(X,d)$ is a geodesic metric space and every two points in $(X,\wSdist)$ are joined by a geodesic \jp.
\end{de}

\begin{rmk}
 If $(X,d)$ is a proper geodesic metric space, then the warped space $(X,\wSdist)$ is jumping\=/geodesic. Note also that if $(X,\wSdist)$ is jumping\=/geodesic then it is a $1$\=/geodesic metric space.
\end{rmk}

For notational convenience, if one of the internal paths $\gamma_i$ of a \jp is constant, we will omit it from the notation and simply write $\jto[s_i]\jto[\smash{s_{i+1}}]$. If the constant path is the initial (or terminal) one, we will keep its value in the notation: $x_0\jto[s_1]\cdots$ (or $\cdots\jto[s_n]x_n$). We will denote the concatenation of \jps simply by $\vec{\bm\gamma}_1\vec{\bm\gamma}_2$. 

If a \jp $\vec{\bm\gamma}=\vec{\bm\gamma}_1\!\jto[\vec s]\!\jto[\hspace{1 pt}\vec s^{\hspace{1 pt}-1}]\vec{\bm\gamma}_2$ has two consecutive jumps such that one is the opposite of the other, we say that the \jp $\vec{\bm\gamma}'=\vec{\bm\gamma}_1\vec{\bm\gamma}_2$ obtained skipping those jumps is a \emph{contraction} of $\vec{\bm\gamma}$. Vice versa, adding a consecutive pair of opposite jumps is an \emph{extension}.

\begin{de}
 Two \jps with the same jumping pattern $\gamma_0\jto[\vec s_1]\cdots\jto[\vec s_n]\gamma_n$ and $\gamma'_0\jto[\vec s_1]\cdots\jto[\vec s_n]\gamma'_n$ are \emph{spatially\=/homotopic} if there exist free homotopies $H_i$ between $\gamma_i$ and $\gamma_i'$ so that for every $t\in[0,1]$ and $i=1,\ldots,n$
 \begin{itemize}
  \item $H_0(0,t)=\gamma_0(0)=\gamma'_0(0)$, 
  \item $H_n(1,t)=\gamma_n(1)=\gamma'_n(1)$, 
  \item $H_i(0,t)=\vec s_i\cdot H_{i-1}(1, t)$.
 \end{itemize}  
 Two \jps are \emph{homotopic} if you can transform one to the other with a finite number of space\=/homotopies, contractions and extensions.
\end{de}

The operation of concatenation between \jps is compatible with homotopies. Given a continuous path $\gamma$, we denote by $\gamma^*$ the path obtained walking along $\gamma$ in the opposite direction. Similarly, if $\vec{\bm \gamma}=\gamma_0\jto[\vec s_1]\cdots\jto[\vec s_n]\gamma_n$ is a \jp, we define $\vec{\bm \gamma}^{*}$ as 
\[
 \vec{\bm \gamma}^{*}\coloneqq \gamma_n^{*}\jto[\vec s_n^{\hspace{1 pt}-1}]\cdots\jto[\vec s_0^{\hspace{1 pt}-1}]\gamma_0^{*}.
\]
Note that the concatenation $\vec{\bm \gamma}\vec{\bm\gamma}^{*}$ is homotopic to the constant \jp; we can therefore give the following:

\begin{de}
 The \emph{jumping\=/fundamental group} is the group $\jpgrp(X,x_0)$ of homotopy classes of closed \jps based at a fixed point $x_0\in X$, equipped with the operation of concatenation.
\end{de}
 
The fundamental observation allowing us to compute jumping\=/fundamental groups is the following:

\begin{lem}\label{lem:paths.and.jumps.commute}
 The \jps $\gamma\jto[\vec s_1]x_1$ and $x_0\jto[\vec s_1]\vec s_1(\gamma)$ are homotopic.
\end{lem}

\begin{proof}
 The functions
 \begin{align*}
  &H_0(r,t)\coloneqq\gamma(r(1-t))  \quad\text{and}\quad H_1(r,t)\coloneqq\vec s_1\big(\gamma(1-t+rt)\big)
 \end{align*}
 define a space\=/homotopy between them.
\end{proof}

\begin{cor}\label{cor:paths.jump.first}
 Every \jp is homotopic to a \jp where all the jumps are performed last: $\gamma\jto[\vec s_1]\jto[\vec s_2]\cdots\jto[\vec s_n]x_0$.
\end{cor}

Let $x_0\in X$ be a fixed base point. For every $s\in S$ choose a continuous path $\alpha_s$ joining $x_0$ to $s\cdot x_0$. Then $\alpha_s\jto[s^{-1}]x_0$ is a closed \jp and we can hence define a homomorphism $\psi_S$ of the free group by extension:
\[
\begin{tikzcd}[row sep=0ex]
\psi_S\colon F_S \arrow[r] & \jpgrp(X,x_0) \\
 \hspace{2 em} s \arrow[r, mapsto] & \bigbrack{\alpha_s\jto[s^{-1}]x_0}.
\end{tikzcd}
\]

Recall that in $\jpgrp(X,x_0)$ we have 
\[
 \Bigbrack{\alpha_s\jto[s^{-1}]x_0}^{-1}=\Bigbrack{\bigparen{\alpha_s\jto[s^{-1}]x_0}^* }=\Bigbrack{ x_0\jto[s]\alpha_s^*}.
\]
Let $ \alpha_{s^{-1}}$ be the path $s^{-1}( \alpha_s^*)$. Then by Lemma~\ref{lem:paths.and.jumps.commute} we have
\[
\Bigbrack{x_0\jto[s]\alpha_s^*} =
\Bigbrack{s^{-1}(\alpha_s^*)\jto[s]x_0} = 
\Bigbrack{ \alpha_{s^{-1}}\jto[s]x_0}
\]
and therefore $\psi_S(s^{-1})=\bigbrack{\alpha_{s^{-1}}\jto[s]x_0}$. We can hence continue to use the notation $\vec s$ to denote an element in $S\cup S^{-1}$ and we have $\psi_S(\vec s)=\bigbrack{ \alpha_{\vec s}\jto[\vec s^{\,-1}]x_0}$. 

\begin{convention*}
 Given a word $w=\vec s_1\cdots\vec s_n$ of elements of $S\cup S^{-1}$, we use the shorthand $\jto[w]$ to denote the concatenation $\jto[\vec s_1]\cdots\jto[\vec s_n]$. We denote by $w_{\rm rev}$ the reverse word $w_{\rm rev}\coloneqq \vec s_n\cdots \vec s_1$, so that $\jto[w_{\rm rev}]$ is short for $\jto[\vec s_n]\cdots\jto[\vec s_1]$.
\end{convention*}

If $w=\vec s_1\cdots\vec s_n$ is a word of elements of $S\cup S^{-1}$, we can define the path $\alpha_w$ in $X$ as the concatenation
\[
 \alpha_w\coloneqq \big(\alpha_{\vec s_1}\big)
    \big(\vec s_{1}(\alpha_{\vec s_{2}})\big) \cdots
    \big(\vec s_{1}\circ \cdots \circ \vec s_{n-1}(\alpha_{\vec s_n})\big).
\]
Note that we have $\alpha_{w_1w_2}=\alpha_{w_1}w_1(\alpha_{w_2})$ and $\alpha_{w^{-1}}=w^{-1}\paren{\alpha_w^*}$.

Lemma~\ref{lem:paths.and.jumps.commute} now implies that $\psi_S(w)$ is the homotopy class $\big[\alpha_w\jto[w_{\rm rev}^{-1}]x_0\big]$.

\begin{rmk}
 Here $w$ is denoting both a word with letters in $S\cup S^{-1}$ and the corresponding element in the free group $F_S$. This ambiguity does not cause any trouble because paths associated with equivalent words are homotopic.
\end{rmk}

The choice of the paths $\alpha_s$ also induces a homomorphism $\phi_S\colon F_S \rightarrow \aut\bigparen{\pi_1(X,x_0)}$ where $\phi_S(s)$ is the automorphism given by
\[
 \phi_S(s)[\gamma]\coloneqq \big[(\alpha_{s}) s(\gamma)(\alpha_{s}^*)\big]
\]
for every $[\gamma]\in\pi_1(X,x_0)$. Just as before, it is easy to check that for every word $w=\vec s_1\cdots\vec s_n$ we have
\[
 \phi_S(w)[\gamma]= \big[\alpha_{w} w(\gamma)\alpha_{w}^*\big].
\]

\begin{thm}\label{thm:jump.fund.grp_iso_semidirect.prod}
The choice of paths $\alpha_s$ induce an isomorphism of groups 
\[
\begin{tikzcd}[row sep=0ex]
  \hspace{2 em}
\Phi_S\colon \pi_1(X,x_0)\rtimes_{\phi_S}F_S \arrow[r] & \jpgrp(X,x_0), \\
\end{tikzcd}
\]
where the map $\Phi_S$ sends $([\gamma],w)$ to $\big[\gamma\alpha_w\jto[w_{\rm rev}^{-1}]x_0\big]$.
\end{thm}

\begin{proof}
The fundamental group $\pi_1(X,x_0)$ is a natural subgroup of $\jpgrp(X,x_0)$ and we already noted that $\psi_S\colon F_S\to\jpgrp(X,x_0)$ is a homomorphism. From Lemma~\ref{lem:paths.and.jumps.commute} it follows that 
 \begin{align*}
  \Bigbrack{\psi_S(s) \gamma\psi_S(s^{-1})} 
    &= \Bigbrack{\alpha_s\jto[s^{-1}] \gamma \alpha_{s^{-1}}\jto[s]x_0 } \\
    &= \Bigbrack{ \alpha_s s(\gamma)s(\alpha_{s^{-1}})\jto[s]\jto[s^{-1}]x_0}\\
    &=  \phi_S(s)[\gamma].
 \end{align*}

 Since $S$ is a generating set, we deduce that $\bigbrack{\psi_S(w)\gamma\psi_S(w^{-1})}=\phi_S(w)\brackets{\gamma}$ for every $w\in F_S$ and hence $\psi_S$ and the inclusion homomorphism induce a homomorphism of groups $\Phi_S\colon \pi_1(X,x_0)\rtimes_{\phi_S}F_S \to \jpgrp(X,x_0)$.
 The expression for $\Phi_S$ follows from the previous discussion, as $\Phi_S([\gamma],w)=\Phi_S([\gamma],e)\Phi_S(e,w)=[\gamma]\psi_S(w)$.
 
 We can promptly show that the map $\Phi_S$ is injective: by our definition of homotopy of \jps, if a \jp $\vec{\bm\gamma}$ is homotopic to a constant path then its jumping pattern must eventually reduce to the trivial word via cancellation of consecutive opposite jumps. Therefore, if $\Phi_S([\gamma],w)=e$, the word $w$ must represent the trivial element in $F_S$; and it is easy to see that $\Phi_S([\gamma],e)=e$ if and only if $[\gamma]$ is also trivial. 
 
 Surjectivity also follows from Lemma~\ref{lem:paths.and.jumps.commute}. Indeed, by Corollary~\ref{cor:paths.jump.first} we know that every \jp is homotopic to a \jp where all the jumps are performed last $\vec{\bm\gamma}=\gamma\jto[\vec s_1]\jto[\vec s_2]\cdots\jto[\vec s_n]x_0$. Then, letting $w=\vec s_1^{\, -1}\cdots\vec s_n^{\, -1}$ we have
 \[
  \bigbrack{\vec{\bm\gamma}}=\bigbrack{\gamma\alpha_w^*\alpha_w\jto[w_{\rm rev}^{-1}]x_0}=\Phi_S\bigparen{[\gamma\alpha_w^*],w}.
 \]
\end{proof}

\section{Discretisations of jumping-paths}\label{sec:discretisations}

From now on we will assume that the parameter $\theta$ is greater than or equal to $1$ and that every continuous path in $(X,d)$ is homotopic to a path of finite length (we say that $(X,d)$ has \emph{homotopy rectifiable paths}). 

Given a continuous path $\gamma\colon [0,1]\to X$, a \emph{$\theta$\=/discretisation} for it is a $\theta$\=/path $\widehat\gamma$ of the form $\bigparen{\gamma(t_0), \gamma(t_1),\ldots,\gamma(t_n)}$ where the times $t_i$ are chosen in such a way that $t_0=0$, $t_n=1$ and $\gamma([t_{i-1},t_i])$ is contained in the closed ball $\overline B\bigparen{\gamma(t_{i-1}),\theta}$ for every $i=1,\ldots,n$.

\begin{rmk}
 It would be notationally more correct to denote $\theta$\=/discretisations by $\widehat\gamma^{\scriptscriptstyle \theta}$ (as we do in \cite{Vig18c}). Still, in most of what follows the parameter $\theta$ will be fixed and hence we decided to drop it from the notation for aesthetic reasons.  
\end{rmk}

It is proved in \cite{Vig18c} (and it is implicit in \cite{BCW14}) that the $\theta$\=/homotopy class of $\widehat\gamma$ does not depend on the specific choice of the times $t_i$. Moreover, if two continuous paths $\gamma$ and $\gamma'$ are (freely) homotopic then any two $\theta$\=/discretisations $\widehat\gamma$ and $\widehat\gamma'$ are (freely) $\theta$\=/homotopic.

Let $F_S\curvearrowright (X,d)$ be a continuous action. Since $\theta\geq 1$, any jump $x\jto[\vec s] \vec s\cdot x$ can be seen as a $\theta$\=/path of length one in $(X,\wSdist)$. We can therefore define the discretisation of a \jp $\vec{\bm \gamma}=\gamma_0\jto[\vec s_1]\cdots\jto[\vec s_n]\gamma_n$ as the concatenation of the discretisations of its continuous pieces and jumps:
\[
 \widehat{\bm \gamma}\coloneqq\widehat\gamma_0\jto[\vec s_1]\cdots\jto[\vec s_n]\widehat\gamma_n.
\]
Note that here the discretisation of the continuous pieces is done using the metric $d$ and not the warped distance $\wSdist$, as we find the geometry of the former more intuitive.

It is still true that the $\theta$\=/homotopy class of a $\theta$\=/discretisation only depends on the homotopy class of the \jp:

\begin{lem}\label{lem:homotopic.paths_homotopic.discretisations}
If two \jps $\vec{\bm\gamma}=\gamma_0\jto[\vec s_1]\cdots\jto[\vec s_n]\gamma_n$ and $\vec{\bm\gamma}'=\gamma_0'\jto[\vec s_1]\cdots\jto[\vec s_m]\gamma_m'$ are homotopic then their discretisations $\widehat{\bm \gamma}$ and $\widehat{\bm \gamma}'$ are $\theta$\=/homotopic.
\end{lem}

\begin{proof}
 It is clear that contractions and expansions of \jps produce $\theta$\=/homotopic \jps. It is therefore enough to prove the statement when $m=n$ and $\vec{\bm\gamma}$ and $\vec{\bm\gamma}'$ are spatially\=/homotopic.

 Let $H_i\colon[0,1]^2\to X$ be the homotopy between $\gamma_i$ and $\gamma_i'$. By compactness, it is easy to see (see \cite{Vig18c} for details) that there exists an $N\in\NN$ large enough so that the maps $\widehat H_i\colon [N]^2\to X$ defined by
 \[
  \widehat H_i(p,q)\coloneqq H_i\left(\frac{p}{N},\frac{q}{N}\right)
 \]
 are free $\theta$\=/grid homotopies between the extremal $\theta$\=/paths $\widehat H_i(\;\!\cdot\;\!,0)\colon [N]\to X$ and $\widehat H_i(\;\!\cdot\;\!,N)\colon[N]\to X$, and moreover the $\theta$\=/paths $\widehat H_i(\;\!\cdot\;\!,0)$ and $\widehat H_i(\;\!\cdot\;\!,N)$ are $\theta$\=/discretisations of $\gamma_i$ and $\gamma_i'$ respectively.
 
 Note that the $\theta$\=/paths $\widehat H_i(\;\!\cdot\;\!,0)$ are concatenated via the jumps $\vec s_i$ (which can be seen as a step of a $\theta$\=/path) and the same goes for the $\theta$\=/paths $\widehat H_i(\;\!\cdot\;\!,N)$ and the $\theta$\=/grid homotopies $\widehat H_i$ as well. 
 Therefore, by concatenation we obtain two $\theta$\=/paths 
 $\widehat H_0(\;\!\cdot\;\!,0)\jto[\vec s_1]\widehat H_1(\;\!\cdot\;\!,0)\jto\cdots\jto\widehat H_N(\;\!\cdot\;\!,0)$ and 
 $\widehat H_0(\;\!\cdot\;\!,N)\jto[\vec s_1]\widehat H_1(\;\!\cdot\;\!,N)\jto\cdots\jto\widehat H_N(\;\!\cdot\;\!,N)$ which are $\theta$\=/homotopic $\theta$\=/discretisations of $\vec{\bm\gamma}$ and $\vec{\bm\gamma}'$. Then the lemma easily follows from the fact that different discretisations of the same continuous path are $\theta$\=/homotopic.
\end{proof}

From Lemma~\ref{lem:homotopic.paths_homotopic.discretisations} it follows that the discretisation procedure induces a well\=/defined discretisation map $\discrmap\colon\jpgrp(X,x_0)\to\thgrp\bigparen{(X,\wSdist),x_0}$. As it is clear that the discretisation of a concatenation of \jps is $\theta$\=/homotopic to the concatenation of their discretisations, the discretisation map $\discrmap$ is a homomorphism.

\

Let $\CT_\theta$ be the set of closed \jps which are compositions of $4$\textemdash non necessarily closed\textemdash \jps of length at most $\theta$:
\[
  \CT_\theta\coloneqq\bigbrace{\vec{\bm\gamma}\bigmid 
    \vec{\bm \gamma}\text{ closed, }\vec{\bm \gamma}=\vec{\bm \gamma}_0\vec{\bm \gamma}_1\vec{\bm \gamma}_2\vec{\bm \gamma}_3
    \text{ with } \norm{\vec{\bm\gamma}_i}\leq\theta
    }
\]
and let ${\rm FT}_\theta\subseteq \jpgrp(X,x_0)$ be the set of homotopy classes of the \jps freely homotopic to \jps in $\CT_\theta$:
\[
 {\rm FT}_\theta\coloneqq\bigbrace{\bigbrack{\vec{\bm\gamma}}\bigmid 
    \vec{\bm\gamma}\frhom \vec{\bm \gamma}'\text{ for some }\vec{\bm\gamma}'\in\CT_\theta},
\]
where a \emph{free homotopy} (denoted $\frhom$) is a homotopy of closed \jps where the space\=/homotopies are not required to keep the endpoints fixed as long as $H_0(0,t)=H_n(1,t)$ for every $t\in[0,1]$. Moreover, cyclic contractions and extensions are allowed as well (\emph{i.e.} $x_0\jto[\vec s]\vec{\bm \gamma}\jto[\vec s^{\, -1}]x_0$ is freely homotopic to $\vec{\bm \gamma}$).

Note that the set ${\rm FT}_\theta$ is invariant under conjugation in $\jpgrp(X,x_0)$.

\begin{thm}\label{thm:discrete.fund.group_iso_quotient.of.jump.fund.grp}
 Assume that $(X,d)$ has homotopy rectifiable paths and that the warped space $(X,\wSdist)$ is jumping\=/geodesic. Then the discretisation homomorphism $\discrmap$ is surjective and its kernel is generated by ${\rm FT}_\theta$. Therefore we have
 \[
  \thgrp\bigparen{(X,\wSdist),x_0}\cong \jpgrp(X,x_0)/\angles{{\rm FT}_\theta}.
 \]
\end{thm}

\begin{proof}
 Let $z\colon [n]\to (X,\wSdist)$ be any $\theta$\=/path. As $(X,\wSdist)$ is jumping\=/geodesic, we can choose for every $i=1,\ldots,n$ a geodesic \jp $\vec{\bm\zeta}_i$ between $z_{i-1}$ and $z_i$. We denote the composition of these paths by $z_{geo}\coloneqq\vec{\bm\zeta}_1\cdots\vec{\bm\zeta}_n$. Note that $z\thhom \widehat z_{geo}$ via the map sending the last point of $\widehat{\bm\zeta}_i$ to $z_i$ and all the others to $z_{i-1}$. This proves the surjectivity of $\discrmap$.
 
 We now have to study the kernel. Let $\vec{\bm\gamma}=\vec{\bm\gamma}_0\vec{\bm\gamma}_1\vec{\bm\gamma}_2\vec{\bm\gamma}_3$ be a \jp in $\CT_\theta$ and for each $0\leq i\leq 3$ let $z_i=\vec{\bm\gamma}_i(0)\in X$ be the starting point of $\vec{\bm\gamma}_i$. Let also $z_4\coloneqq z_0$; then the sequence $(z_0,\ldots,z_4)$ is a closed $\theta$\=/path and it is easy to see that it is $\theta$\=/homotopic to a constant path as
 \[
  (z_0,z_1,z_2,z_3,z_4=z_0)\thhom(z_0,z_1,z_1,z_0,z_0)\thhom(z_0,z_0,z_0,z_0,z_0)
 \]
 are all $\theta$\=/grid homotopies. As above, we also have $\widehat{\bm\gamma}_i\thhom(z_i,z_{i+1})$ and therefore the discretised path $\widehat{\bm\gamma}$ is itself $\theta$\=/homotopic to the constant path.
 
 If a \jp $\vec{\bm\gamma}$ is freely homotopic to $\vec{\bm\gamma}'\in\CT_\theta$, tracing the movement of the base point under the free homotopy we obtain a \jp $\vec{\bm\beta}$ joining $x_0$ to $\vec{\bm\gamma}'(0)$ and we see that $\vec{\bm\gamma}$ is genuinely homotopic to $\vec{\bm\beta}\vec{\bm\gamma}'\vec{\bm\beta}^*$. We deduce that 
 \[
  \widehat{\bm\gamma}\thhom \widehat{\bm\beta}\widehat{\bm\gamma}'\widehat{\bm\beta}^*
    \thhom \widehat{\bm\beta}\widehat{\bm\beta}^* \thhom x_0
 \]
 and therefore ${\rm FT}_\theta\subseteq\ker(\widehat\Phi_S)$.

For the inverse inclusion, we begin by showing that if two $\theta$\=/paths are $\theta$\=/homotopic via a $1$\=/step $\theta$\=/grid homotopy, then their `geodesifications' differ by a product of paths in $\rm TF_\theta$. Specifically, let $z,z'\colon [n]\to (X,\wSdist)$ be closed $\theta$\=/paths with base point $x_0$ and so that $\wSdist(z_i,z'_i)\leq\theta$ for every $i\in[n]$. 
As above, let $z_{geo}\coloneqq\vec{\bm\zeta}_1\cdots\vec{\bm\zeta}_n$ and $z'_{geo}\coloneqq\vec{\bm\zeta}'_1\cdots\vec{\bm\zeta}'_n$ be concatenations of geodesic \jps. Choose geodesic \jps $\vec{\bm\varepsilon}_i$ joining $z_i$ to $z'_i$ and let $ \vec{\bm\xi}_i\coloneqq\vec{\bm\zeta}^*_i\vec{\bm\varepsilon}_{i-1}\vec{\bm\zeta}'_i\vec{\bm\varepsilon}_{i}^*$; then the \jp $z_{geo}'$ is homotopic to the composition $\vec{\bm\zeta}_1\vec{\bm\xi}_1\cdots\vec{\bm\zeta}_n\vec{\bm\xi}_n$ (see Figure~\ref{fig:geodesifications.of.homotopy}).

\begin{figure}
 \centering
   \begin{tikzpicture}[scale=0.98]
    \draw [thick,
	    postaction={decorate}, decoration={markings,
	      mark=between positions 0.1 and 0.95 step 0.204 with {\arrow{stealth};}}
	  ]
	  (0,0) ..controls (1.5,-0.9) and (6.5,-0.9).. (8,0)
	  node[pos=0,bnode,inner sep=0.6]{} node[pos=0,left]{$x_0=z_0=z_0'$}
	  node[pos=0.23,bnode,inner sep=0.6](l1){}node[pos=0.23,below]{$z_1$}
	  node[pos=0.11,below]{$\vec{\bm\zeta}_1$}
	  node[pos=0.41,bnode,inner sep=0.6](l2){}node[pos=0.41,below]{$z_2$}
	  node[pos=0.32,below]{$\vec{\bm\zeta}_2$}
	  node[pos=0.59,bnode,inner sep=0.6](l3){}node[pos=0.59,below]{$\vphantom{z}\cdots$}
	  node[pos=0.77,bnode,inner sep=0.6](l4){}node[pos=0.77,below]{$z_{n-1}$}
	  node[pos=0.89,below]{$\vec{\bm\zeta}_n$}
	  node[pos=1,bnode,inner sep=0.6]{}node[pos=1,right]{$x_0=z_n=z_n'$};
    
    \draw [thick,
	    postaction={decorate}, decoration={markings,
	      mark=between positions 0.1 and 0.95 step 0.204 with {\arrow{stealth};}}
	  ]
	  (0,0) ..controls (1.5,0.9) and (6.5,0.9).. (8,0)
	  node[pos=0,bnode,inner sep=0.6]{}
	  node[pos=0.23,bnode,inner sep=0.6](u1){}node[pos=0.23,above]{$z'_1$}
	  node[pos=0.11,above]{$\vec{\bm\zeta}_1'$}
	  node[pos=0.41,bnode,inner sep=0.6](u2){}node[pos=0.41,above]{$z'_2$}
	  node[pos=0.32,above]{$\vec{\bm\zeta}_2'$}
	  node[pos=0.59,bnode,inner sep=0.6](u3){}node[pos=0.59,above]{$\vphantom{z}\cdots$}
	  node[pos=0.77,bnode,inner sep=0.6](u4){}node[pos=0.77,above]{$z'_{n-1}$}
	  node[pos=0.89,above]{$\vec{\bm\zeta}_n'$}
	  node[pos=1,bnode,inner sep=0.6]{};
    
    \draw [middlearrow=stealth] (l1.center)..controls +(0.1,0.3) and +(0.1,-0.3)..(u1.center)node[pos=0.5,left]{$\vec{\bm\varepsilon}_1$}
	(u1.center)node[pos=0.5,xshift=0.8 cm]{$\vec{\bm\xi}_2$};
    \draw [middlearrow=stealth] (l2.center)..controls +(0.1,0.3) and +(0.1,-0.3)..(u2.center)node[pos=0.5,right]{$\vec{\bm\varepsilon}_2$};
    \draw [middlearrow=stealth] (l4.center)..controls +(0.1,0.3) and +(0.1,-0.3)..(u4.center)node[pos=0.5,left]{$\vec{\bm\varepsilon}_{n-1}$};
    
      \path (l2) +(-0.1,0.1)node(x1){}
	  (l1) +(0.1,0.1)node(x2){}
	  (u1) +(0.1,-0.1)node(x3){}
	  (u2) +(-0.1,-0.1)node(x4){};
      \draw [thick, densely dotted,->,>=stealth,rounded corners] (l2.north west)--
      (x2.center)..controls +(0.05,0.25) and +(0.05,-0.25)..
      (x3.center)--
      (x4.center) ..controls +(0.05,-0.25) and +(0.05,0.25)..
      (x1.center);
   \end{tikzpicture}
   \caption{Geodesifications filling in a discrete homotopy.}
   \label{fig:geodesifications.of.homotopy}
\end{figure}
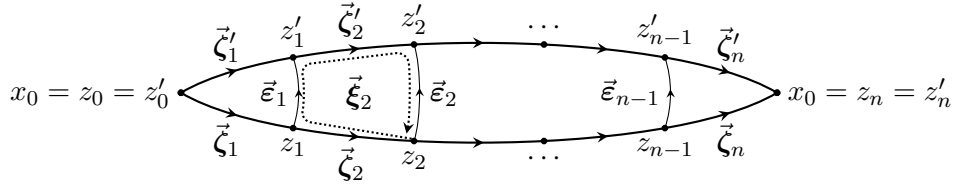

Let now $\vec{\bm\beta}_i$ be a \jp joining $x_0$ to $z_i$; then in $\jpgrp(X,x_0)$ we have
\[
 [z_{geo}]=\Bigbrack{\vec{\bm\zeta}_1\cdots\vec{\bm\zeta}_n}
    =\Bigbrack{\vec{\bm\zeta}_1\vec{\bm\beta}_1^*}\Bigbrack{\vec{\bm\beta}_1\vec{\bm\zeta}_2\vec{\bm\beta}_2^*}\cdots
	\Bigbrack{\vec{\bm\beta}_{n-1}\vec{\bm\zeta}_n}
\]
and
\[
 [z'_{geo}]=\Bigbrack{\vec{\bm\zeta}'_1\cdots\vec{\bm\zeta}'_n}
    =\Bigbrack{\vec{\bm\zeta}_1\vec{\bm\beta}_1^*}\Bigbrack{\vec{\bm\beta}_1\vec{\bm\xi}_1\vec{\bm\beta}_1^*}
      \Bigbrack{\vec{\bm\beta}_1\vec{\bm\zeta}_2\vec{\bm\beta}_2^*}\cdots
      \Bigbrack{\vec{\bm\beta}_{n-1}\vec{\bm\zeta}_n}\Bigbrack{\vec{\bm\xi}_n}.
\]
Note that $\vec{\bm\xi}_i\in\CT_\theta$, and therefore $\vec{\bm\beta}_i\vec{\bm\xi}_i\vec{\bm\beta}_i^*\in{\rm FT}_\theta$.
As ${\rm FT}_\theta$ is invariant under conjugation, it follows that $[z_{geo}]\equiv[z'_{geo}] \pmod{\rm FT_\theta}$.

Let $\vec{\bm\gamma}$ be any closed \jp and $\widehat{\bm\gamma}\colon[n]\to X$ its discretisation; we claim that $\bigbrack{\vec{\bm\gamma}}\equiv\bigbrack{\widehat{\bm\gamma}_{geo}} \pmod{\rm FT_\theta}$\textemdash note that we do not claim that $\vec{\bm\gamma}$ and $ \widehat {\bm\gamma}_{geo}$ be homotopic. By hypothesis we can assume that $\vec{\bm\gamma}$ is composed of rectifiable paths. Let $\beta_{j-1}^j$ be the sub\=/path of $\vec{\bm\gamma}$ going from $\widehat{\bm\gamma}(j-1)$ to $\widehat{\bm\gamma}(j)$\textemdash it could either be a continuous path or a single jump. We will now argue as above to conclude that $\beta_{j-1}^j$ and the geodesic between $\widehat{\bm\gamma}(j-1)$ and $\widehat{\bm\gamma}(j)$ differ by a product of loops in $\CT_\theta$. This is clearly doable if $\beta_{j-1}^j$ is a single jump; while if $\beta_{j-1}^j$ is a continuous path it is can be subdivided in finitely many pieces of length at most $\theta$ (because it is rectifiable) and the argument can be applied when seeing $\beta_{j-1}^j$ as the concatenation of these smaller pieces (see Figure \ref{fig:subdividing.geodesic}). 

\begin{figure}
 \centering
 \begin{tikzpicture}
  \path[clip](-2,-1.8) rectangle (5,3.5);
  \draw[dotted] (0,0) circle (4);
  \draw[thick,dashed](0,0) node[bnode,inner sep=0.7]{}node[left]{$\widehat{\bm \gamma}(j-1)$} 
      --(4,0)node[bnode,inner sep=0.7]{}node[right]{$\widehat{\bm \gamma}(j)$}
      (2.3,2)node{$\beta_{j-1}^{j}$};
  \draw plot[smooth, tension=0.7] coordinates { (0,0)(0.1,2)(0.3,2)(0.5,-1.5)(1.5,2)(2.5,-1.7)(3,1.2)(4,0) }
	  [postaction={decorate}, 
	  decoration={markings, mark=between positions 0.09 and 0.99 step 0.12 with {\draw (0,-2 pt) --(0,2pt);}},
	  ];
  \end{tikzpicture}
 \caption{Subdividing a piece of a continuous paths.}
 \label{fig:subdividing.geodesic}
\end{figure}
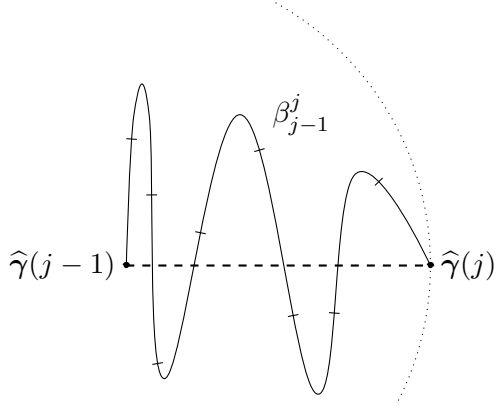

Assume now that $\vec{\bm\gamma}$ and $\vec{\bm\gamma}'$ are \jps with $\theta$\=/homotopic discretisations; then there exists a $\theta$\=/grid homotopy $\widehat H\colon [n]\times [m]\to X$ between lazified versions of $\widehat{\bm\gamma}$ and $\widehat{\bm\gamma}'$. Let $\widehat{\bm\gamma}^{(k)}\coloneqq \widehat H(\;\!\cdot\;\!,k)$, note that $\widehat{\bm\gamma}_{geo}=\widehat{\bm\gamma}^{(0)}_{geo}$ and $\widehat{\bm\gamma}'_{geo}=\widehat{\bm\gamma}^{(m)}_{geo}$ as lazifying does not modify the actual paths. 
By the above discussion, for every $k=1,\ldots,m$ we have $\bigbrack{\widehat{\bm\gamma}^{(k-1)}_{geo}}\equiv\bigbrack{\widehat{\bm\gamma}^{(k)}_{geo}}\pmod{\rm FT_\theta}$ and, as we know that $\bigbrack{\vec{\bm\gamma}}\equiv \bigbrack{\widehat{\bm\gamma}_{geo}}$ and $\bigbrack{\vec{\bm\gamma}'}\equiv\bigbrack{ \widehat{\bm\gamma}'_{geo}} \pmod{\rm FT_\theta}$, this concludes the proof of the theorem.
\end{proof}

\begin{rmk}
This is a sort of continuous version of \cite[Theorem 2.7]{BKLW01}. The techniques we use here are also similar in spirit to what we use in the proof of the main result of \cite{Vig18c}, but we had to include a proof because there are some additional difficulties.

Note also that this is the only place where we really need the geodesicity hypothesis. It is clear from the proof that some version of this theorem can be proved with weaker hypotheses (\emph{e.g.} spaces with path metrics). We decided not to do so to avoid unnecessary complications. 
\end{rmk}

\section{Criteria for explicit computations}\label{sec:criteria.for.computations}

From now on we will assume $(X,d)$ to have homotopic rectifiable paths and $(X,\wSdist)$ to be jumping\=/geodesic. Combining Theorem~\ref{thm:discrete.fund.group_iso_quotient.of.jump.fund.grp} with Theorem~\ref{thm:jump.fund.grp_iso_semidirect.prod} we obtain a (non\=/canonical, as it depends on the choice of paths $\alpha_s$) surjection 
\[
\begin{tikzcd}[row sep=0ex]
\widehat\Phi_S\colon\pi_1(X,x_0)\rtimes_{\phi_S}F_S \arrow[two heads]{r} & \thgrp\bigparen{(X,\wSdist),x_0}.
\end{tikzcd}
\] 
We now wish to study the kernel of $\widehat\Phi_S$ more explicitly. 

To begin with, let $\shpa\coloneqq\bigbrace{[\gamma]\in\pi_1(X,x_0)\bigmid \gamma\frhom\gamma',\ \norm{\gamma'}\leq 4\theta}$ be the set of continuous loops based at $x_0$ that are freely homotopic to continuous loops of length at most $4\theta$ in $(X,d)$. It is then clear that $\shpa\times \{e\}$ is in the kernel of $\widehat \Phi_S$. Moreover, as it is a subset of the normal factor of the semidirect product, taking the quotient by $\aangles{\shpa\times \{e\}}$ preserves the structure of semidirect product. Therefore $\widehat\Phi_S$ factors as
\[
\begin{tikzcd}
\pi_1(X,x_0)\rtimes_{\phi_S}F_S \arrow{r}{\widehat\Phi_S}\arrow{d} & \thgrp\bigparen{(X,\wSdist),x_0} \\
\pi_1(X,x_0)/\aangles{\shpa}\rtimes_{\phi_S}F_S \arrow{ru} &
\end{tikzcd}
\] 
where the normal closure $\aangles{\shpa}$ is taken in the whole semidirect product and therefore we have
\[
 \aangles{\shpa}=\angles*{\vphantom{\big(}\braces{\phi_S(w)[\gamma]\mid w\in F_S,\ [\gamma]\in\shpa}}.
\]
Equivalently, $\aangles{\shpa}$ is the subgroup generated by the set of continuous paths which are freely homotopic to the image under some $w\in F_S$ of a short closed path. 

Let $\shjp\coloneqq\bigbrace{w\in\ F_S\bigmid \exists y\in X,\ d(y,w\cdot y)+\abs{w}\leq 4\theta}$ be the set of elements $w\in F_S$ that move some point $y\in(X,d)$ by less than $4\theta$ minus the length of the reduced word of $w$.

\begin{lem}\label{lem:short.jumps.are.trivial}
 Let $(X,d)$ be a geodesic metric space and let $\theta$ be a natural number. Then for every $w\in \shjp$ there exists a $[\gamma]\in\pi_1(X,x_0)$ so that $([\gamma],w)\in\ker\bigparen{\widehat\Phi_S}$.
\end{lem}

\begin{proof}
 Let $y$ be a point so that $d(y,w\cdot y)+\abs{w}\leq 4\theta$, $\beta$ be a continuous path from $x_0$ to $y$, and let $\varepsilon$ be a continuous geodesic path from $y$ to $w\cdot y$. Note that, as $\theta$ is an integer, the closed \jp $\varepsilon\jto[w^{-1}_{\rm rev}]y$ can be decomposed into four sub\=/paths of length at most $\theta$, and it is therefore in $\CT_\theta$.
 
 By Lemma~\ref{lem:paths.and.jumps.commute} we have
 \[
  \bigbrack{\alpha_w\jto[w^{-1}_{\rm rev}]x_0}
  =\bigbrack{\alpha_ww(\beta)\jto[w^{-1}_{\rm rev}]\beta^*}
  =\bigbrack{\alpha_w w(\beta)\varepsilon^*\beta^*}\bigbrack{\beta\varepsilon\jto[w^{-1}_{\rm rev}] \beta^*},
 \]
 and since $\bigbrack{\beta\varepsilon\jto[w^{-1}_{\rm rev}] \beta^*}$ is in ${\rm FT}_\theta$ we can conclude the proof by letting $\gamma\coloneqq\beta\varepsilon w(\beta^*)\alpha_w^*$ (see Figure~\ref{fig:short.jumps.are.trivial}).
\end{proof}

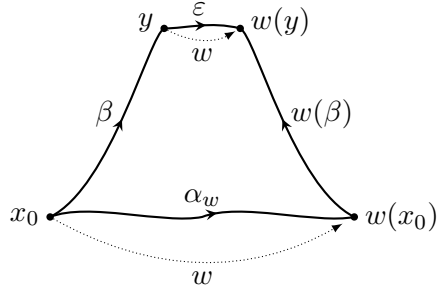
\begin{figure}
 \centering
 \begin{tikzpicture}
  \draw[thick, middlearrow=stealth](0,0)node[bnode,inner sep=0.7]{}node[left]{$x_0$} 
	  ..controls(0.5,0.2) and (1.5,-0.1).. 
	  (2,0)node[above]{$\alpha_w$} ..controls(2.5,0.2) and (3.5,-0.1).. 
	  (4,0)node[bnode,inner sep=0.7]{}node[right]{$w(x_0)$};
  \draw[thick, middlearrow=stealth](1.5,2.5) node[bnode,inner sep=0.7]{}node[yshift=0.5 ex, left]{$y$}
	  ..controls(1.8,2.5) and (2.1,2.6).. 
	  (2.5,2.5)node[pos=0.5,above]{$\varepsilon$}node[bnode,inner sep=0.7]{}node[yshift=0.5 ex,right]{$w(y)$};
  \draw[thick, middlearrow=stealth](0,0) ..controls(0.8,0.4) and (1.3,2.4).. 
	  (1.5,2.5)node[pos=0.5,left]{$\beta$};
  \draw[thick, middlearrow=stealth] (4,0)..controls  (3.2,0.4) and (2.7,2.4) .. 
	  (2.5,2.5)node[pos=0.5,right]{$w(\beta)$};

  \draw [densely dotted, bend right,->,shorten >= 5 pt](0,0) to (4,0) ;
  \draw [densely dotted, bend right,->,shorten >= 2 pt](1.5,2.5) to (2.5,2.5);
  \draw (2,-0.8)node{$w$} (2,2.15)node{$w$};
  \end{tikzpicture}
 \caption{Constructing a continuous path that trivializes a short jump $w$.}
 \label{fig:short.jumps.are.trivial}
\end{figure}

\begin{cor}\label{cor:S.in.shortjumps_thgrp.quotient.of.pi}
 If $(X,d)$ is geodesic, $\theta\in\NN$ and $s\in\shjp$ for every $s\in S$, then $\thgrp\bigparen{(X,\wSdist),x_0}$ is isomorphic to a quotient of $\pi_1(X,x_0)$.
\end{cor}

The converse of Lemma \ref{lem:short.jumps.are.trivial} is not true in general. Still, it does hold if the action is by isometries.

\begin{lem}
 If the action $F_S\curvearrowright (X,d)$ is by isometries, then there exists a $\gamma$ so that $([\gamma],w)\in\ker\bigparen{\widehat\Phi_S}$ only if $w\in\aangles{\shjp}$.
\end{lem}

\begin{proof}
 Let $\vec{\bm \gamma}=\gamma_1\jto[\vec s_1]\cdots\jto[\vec s_n]\gamma_n$ be a \jp and let $w=\vec s_n^{\,-1}\cdots \vec s_1^{\,-1}$. As the action is by isometries, the \jp $\gamma'\jto[w^{-1}_{\rm rev}]x_0$ homotopic to $\vec{\bm\gamma}$ as in Corollary \ref{cor:paths.jump.first} has the same length of $\vec{\bm\gamma}$. It follows that if $\vec{\bm \gamma}$ is (freely homotopic to) a path in $\CT_\theta$ then $w$ is (conjugate to) an element in $\shjp$.
\end{proof}

\begin{cor}
 If $F_S$ acts by isometries and $\shjp=\{e\}$, then 
 \[
\thgrp\bigparen{(X,\wSdist),x_0}\cong \pi_1(X,x_0)/\aangles{\shpa}\rtimes_{\phi_S}F_S.
 \]
\end{cor}

\begin{proof}
 Every \jp in $\CT_\theta$ must be continuous as it cannot have any jump. It follows that the kernel of $\widehat\Phi_S$ is equal to $\aangles{\shpa}$.
\end{proof}

So far we have focused our attention to warped metrics coming from actions of (finitely generated) free groups. This is the most general setting as any action $\rho$ of a (finitely generated) group $\Gamma$ can be seen as a quotient of a non\=/faithful action of a (finitely generated) free group by choosing a generating set $S$ for $\Gamma$ and considering the action induced by $\rho$ on $F_S$. Still, the description of $\thgrp$ that we obtain this way is not completely satisfactory. Indeed, one expects the discrete fundamental group to encode some information on the large scale geometry of the metric space, and the coarse geometry of warped cones depends on $\Gamma$ and not on the choice of generating set $S$. The next result tries to uncover the dependence of $\thgrp$ on $\Gamma$, but in order to do so we first need to prove the following lemma:

\begin{lem}\label{lem:quotient.of.semidirect.prod}
 Let $G\rtimes_\phi H$ be a semidirect product and $N\lhd H$ a normal subgroup. Assume that there exists a function $f\colon N\to G$ such that
 \[
  \phi(n)(g)=g^{f(n)}
 \]
for every $n\in N$ and $g\in G$\textemdash where $g^{f(n)}\coloneqq f(n)nf(n)^{-1}$ denotes the conjugation by $f(n)$\textemdash and $f$ is $\phi$\=/equivariant (in the sense that $\phi(h)(f(n))=f(n^h)$ for every $h\in H$). Define $Q$ to be the quotient
\[
 Q\coloneqq \frac{\bigparen{G\rtimes_\phi H}}{\aangles{\bigparen{f(n)^{-1},n}\mid n\in N} };
\]
then there is a natural short exact sequence
 \[
  1\to \ \aangles{f(N)}_Q \to\ Q \ \xrightarrow{\bar p}\  \frac{G}{\aangles{f(N)\cup[f(N),G]} }\rtimes_{\bar\phi} \frac{H}{N}\ \to 1  
 \]
 where $[f(N),G]$ denotes the set $\bigbrace{g^{f(n)}g^{-1}\mid n\in N,\ g\in G}$ and $\aangles{f(N)}_Q$ is the normal subgroup generated by the image of $f(N)$ in $Q$.
\end{lem}

\begin{proof}
 We first need to show that the semidirect product on the right hand side is well\=/defined. For every $g\in G$, $n\in N$ and $h\in H$ we have 
 \[
  \phi(h)\bigparen{f(n)}=f\bigparen{n^h}
 \]
and
 \[
  \phi(h)\bigparen{g^{f(n)}g^{-1}}=\bigparen{\phi(h)(g)}^{\phi(h)(f(n))}\phi(h)(g)^{-1}
    =\bigparen{\phi(h)(g)}^{f(n^h)}\phi(h)(g)^{-1}.
 \]
Since $N$ is a normal subgroup of $H$, we deduce that $\aangles{f(N)\cup[f(N),G]}$ is preserved by $\phi$ and we therefore obtain an $H$\=/action on the quotient $G/\aangles{f(N)\cup[f(N),G]}$.
Moreover, $N$ acts trivially on $G/\aangles{f(N)\cup[f(N),G]}$ as $\bigparen{\phi(n)(g)}g^{-1}$ is in $[f(N),G]$ by definition. It follows that $\phi$ naturally induces the required homomorphism
\[
 \bar\phi\colon H/N\to \aut\Bigparen{\frac{G}{\aangles{f(N)\cup[f(N),G]}} }
\]
and we obtain a natural surjection
\[
 p\colon\bigparen{G\rtimes_\phi H}\ \longrightarrow \ \frac{G}{\aangles{f(N)\cup[f(N),G]} }\rtimes_{\bar\phi} \frac{H}{N}
\]
sending $(g,h)$ to $(\bar g,\bar h)$. It only remains to study the kernel of such a surjection. 

The elements $\bigparen{f(n)^{-1},n}$ trivially belong to the kernel of $p$ for every $n\in N$ and therefore $p$ factors through the quotient $Q$. We obtain this way a surjective homomorphism $\bar p$ of $Q$ onto the semidirect product. It only remains to show that the kernel of $\bar p$ is exactly $\aangles{f(N)}_Q$. Note that one containment is obvious, as $\bigparen{f(n),e}$ is in $\ker(p)$ for every $n\in N$.

Vice versa, if $(g,h)$ belongs to $\ker(p)$ then $h=n$ for some $n\in N$, so that $(g,h)\equiv \bigparen{g f(n),e} \mod \aangles{\bigparen{f(n)^{-1},n}\mid n\in N}$. We can hence restrict to the study of elements of $\ker(p)$ of the form $(g,e)$. Given such an element, $g$ can be expressed as a product of conjugates of $f(n)^{\pm 1}$ or $\bigparen{g^{f(n)}g^{-1}}^{\pm 1}$ with $n\in N$ and $g\in G$. It is hence enough to observe that (the equivalence class of) a conjugate $f(n)^{a}$ is in $\aangles{f(N)}_Q$ by definition and that
\[
 \bigparen{g^{f(n)}g^{-1}}^{a}=\Bigparen{ f(n) \bigparen{f(n)^{-1}}^{g} }^{a}
\]
is in $\aangles{f(N)}_Q$ as well.
\end{proof}

Recall that a \emph{presentation} for a group $\Gamma$ is a choice of generating set $S$ and a set $R$ of words in $S\cup S^{-1}$ (called relations) such that $\aangles{R}$ is the kernel of the natural surjection $F_S\to\Gamma$. Such a presentation is denoted as $\Gamma=\angles{S\mid R}$, and a group said to be \emph{finitely presented} if it admits a presentation where $S$ and $R$ are finite. In what  follows, we let $\Gamma_\theta$ be the finitely presented group $\angles{S\mid R_\theta}$ where $R_\theta$ is the subset of $\aangles{R}$ of words of length at most $4\theta$.

Note that if the group $\Gamma$ is acting by homeomorphisms on $(X,d)$ and $r\in \aangles{R}$ is a relation of $\Gamma$, then the path $\alpha_r$ is closed and it hence defines an element in $\pi_1(X,x_0)$. We can now prove the following.

\begin{prop}\label{prop:discr.fund.group.is.semidirect.Gamma}
 Let $\Gamma=\angles{S\mid R}$ be a finitely generated group acting on $(X,x_0)$ by homeomorphisms, fix $\theta\in \NN$ and let $G_\theta$ be the quotient
 \[
  G_\theta\coloneqq\frac{\pi_1(X,x_0)}{\aangles{\bigbrace{[\alpha_r]\bigmid r\in \aangles{R_\theta}}\cup \bigbrace{[\alpha_r\gamma\alpha_r^*\gamma^{-1}]\bigmid [\gamma]\in \pi_1(X,x_0),\ r\in \aangles{R_\theta}} } }.
 \]
 Then the quotient $\thgrp\big((X,\delta_S),x_0\big)/\aangles{\widehat\Phi_S([\alpha_r^*],e)\mid r\in R_\theta}$ is isomorphic to a quotient of $G_\theta\rtimes \Gamma_\theta$.
\end{prop}
\begin{proof}
 This will be an application of Lemma~\ref{lem:quotient.of.semidirect.prod}, where $\pi_1(X,x_0)$ and $F_S$ play the role of $G$ and $H$ respectively, and $N$ corresponds to $\aangles{R_\theta}$. Note that the function sending $r\in \aangles{R_\theta}$ to $[\alpha_r]\in\pi_1(X,x_0)$ satisfies
 \[
  \phi_S(r)[\gamma]=[\alpha_rr(\gamma)\alpha_r^*]=[\alpha_r][\gamma][\alpha_r]^{-1}=[\gamma]^{[\alpha_r]}
 \]
 and it is also $\phi_S$\=/equivariant, as we have
 \begin{align*}
  \phi_S(w)[\alpha_{r}] &=\bigbrack{\alpha_ww(\alpha_r)\alpha_w^*} \\
  &=\bigbrack{\alpha_ww\bigparen{\alpha_rw^{-1}(\alpha_w^*)}} \\
  &=\bigbrack{ \alpha_ww\bigparen{\alpha_{r}r(\alpha_{w^{-1}})} } = \bigbrack{\alpha_{wrw^{-1}} }
 \end{align*}
 (here we used essentially that $r$ acts as the identity on $X$). We can hence apply Lemma~\ref{lem:quotient.of.semidirect.prod} to obtain a short exact sequence
 \[
  1\to \ \aangles{[\alpha_r]\mid r\in R_\theta}_Q  \to\ 
  \frac{\pi_1(X,x_0)\rtimes_{\phi_S}F_S}{\aangles{\bigbrace{([\alpha_r^*],r)\bigmid r\in R_\theta}} } 
  \ \xrightarrow{\bar p}\  
  \bigparen{G_\theta\rtimes_{\bar\phi} \Gamma_\theta} \to 1 . 
 \]
 
 Note that in the expression above we used the fact that $[\alpha_r]^{-1}=[\alpha_r^*]$, $\alpha_{r_1r_2}=\alpha_{r_1}\alpha_{r_2}$ and that $\alpha_{wrw^{-1}}=\alpha_ww(\alpha_{r})\alpha_{w}^*=\phi_S(w)(\alpha_r)$ to deduce that
 \[
   \aangles{[\alpha_r]\mid r\in \aangles{R_\theta}}_Q= \aangles{[\alpha_r]\mid r\in R_\theta}_Q
 \]
 and
 \[
  \aangles{\bigbrace{([\alpha_r]^{-1},r)\bigmid r\in \aangles{R_\theta} }}
  =\aangles{\bigbrace{([\alpha_r^*],r)\bigmid r\in R_\theta}}.
 \]
 However, in general it is not possible to restrict to $r\in R_\theta$ in the definition of $G_\theta$, as $\alpha_{wrw^{-1}}$ is conjugate to $\alpha_r$ only in $\pi_1(X,x_0)\rtimes F_S$, and not in $\pi_1(X,x_0)$.
 
 The homomorphism $\widehat\Phi_S$ is a surjection of the group $\pi_1(X,x_0)\rtimes F_S$ onto $\thgrp\bigparen{(X,\wSdist),x_0}$. Given $r\in R_\theta$, we have 
 \[
  \widehat\Phi_S\bigparen{([\alpha_r^*],r)}
  =\bigbrack{\alpha_r^*\alpha_r\jto[r_{\rm rev}^{-1}]x_0}
  =\bigbrack{x_0\jto[r_{\rm rev}^{-1}]x_0}
 \]
 and the latter is trivial as it is a closed $\theta$\=/path of length at most $4\theta$. This implies that $\widehat\Phi_S$ factors through the quotient by $\aangles{\bigbrace{([\alpha_r^*],r)\bigmid r\in R_\theta}}$.
 
 To conclude the proof it is enough to note that $\widehat\Phi_S$ also factors to a surjection 
 \[
  \bigparen{G_\theta\rtimes_{\bar\phi} \Gamma_\theta} \to\
  \frac{\thgrp\big((X,\delta_S),x_0\big)}{\aangles{\widehat\Phi_S([\alpha_r^*],e)\mid r\in R_\theta} }.
 \]
\end{proof}

\begin{cor}
 If the paths $\alpha_r$ are null\=/homotopic in $X$ for every $r\in R$, then $\thgrp\bigparen{(X,\delta_S),x_0}$ is isomorphic to a quotient of $\pi_1(X,x_0)\rtimes \Gamma_\theta$.
\end{cor}

\section{Computing the discrete fundamental group of warped cones}\label{sec:coarse.fund.grp.wc}

Here we will continue to assume that $X$ has homotopy rectifiable paths, and, for our main result, we will also need to assume that $X$ is compact.

Let $F_S\curvearrowright (X,d)$ be a continuous action and recall that $\wSdist[t]$ denotes the rescaled metric $\dist[t]$ on $X$ warped by $F_S$. Note that if $X$ is compact and $(X,\wSdist)$ is jumping\=/geodesic, then also $(X,\wSdist[t])$ is jumping\=/geodesic for every $t\geq 1$. Indeed, by compactness, for every $x,x'\in X$ there exist $y_1,\ldots,y_n\in X$ and $\vec s_1\ldots \vec s_n\in S\sqcup S^{-1}$ so that
\[
 \wSdist[t](x,x')=d^t(x,y_1)+1+d^t\bigparen{\vec s_1(y_1),y_2}+1+\cdots+d^t\bigparen{\vec s_n(y_n),x'}.
\]
As there is a jumping geodesic in $(X,\wSdist)$ between $x$ and $y_1$, we deduce that that jumping\=/geodesic must be an actual continuous geodesic in $(X,d)$ and hence the same path gives us a geodesic in $(X,d^t)$. The same argument holds for all the contributions of $d^t$, and we can thus build a jumping\=/geodesic between $x$ and $x'$ in $(X,\wSdist[t])$.

Let ${\rm Ell}_\theta\coloneqq\bigbrace{w\in F_S\bigmid \abs{w}\leq 4\theta,\ \fix(w)\neq \emptyset}$ be the set of elliptic elements of length at most $4\theta$ (in general this set is not closed under conjugation because of the condition $\abs{w}\leq 4\theta$). Note that if $y\in X$ is a fixed point of $w$ and $\beta$ is a continuous path joining $x_0$ to $y$, then $\alpha_w w(\beta)\beta^*$ is a closed loop in $X$. 

\begin{thm}\label{thm:disc.fund.group.of.warped.cones}
 Let $X$ be compact, $(X,\wSdist)$ jumping\=/geodesic and fix $\theta\in\NN$. Then there exists a $t_0$ large enough so that for every $t\geq t_0$ and $w\in F_S$ there exists a path $\gamma$ so that $([\gamma],w)\in\ker\bigparen{\widehat\Phi_S\colon\pi_1(X)\rtimes_{\phi_S} F_S\to\thgrp(X,\wSdist[t])}$ if and only if $w\in\aangles{\rm Ell_\theta}$. 
 
 Moreover, if $X$ is semi\=/locally simply connected we can choose $t_0$ large enough so that for every $t\geq t_0$ we have
 \[
  \thgrp\bigparen{(X,\wSdist[t]),x_0}\cong\bigparen{\pi_1(X,x_0)\rtimes_{\phi_S}F_S}\big/\aangles{K_\theta}
 \]
 where 
 \[
  K_\theta\coloneqq\bigbrace{\bigparen{[\beta w(\beta^*)\alpha_w^*],w}\bigmid w\in{\rm Ell}_\theta,\ \beta(0)=x_0,\ \beta(1)\in\fix(w)}.
 \]
\end{thm}

\begin{proof}
 Following the proof of Lemma \ref{lem:short.jumps.are.trivial} it is easy to see that ${K_\theta}\subseteq\ker{\widehat\Phi_S}$ for $t$ large enough. Therefore, for every $w\in\aangles{\rm Ell_\theta}$ we have explicitly exhibited the required path $\gamma$ so that $([\gamma],w)$ is in the kernel.
 
 We now prove the converse. Given a word $w=\vec s_1\cdots \vec s_{\abs{w}}$, a point $x\in X$ and a radius $r\geq 0$, we define a sequence of sets as follows: $C_w^{(0)}(x,r)$ is the ball $B(x,r)$ in $(X,\dist)$ and for $1\leq i\leq \abs{w}$ we let
 \[
  C_w^{(i)}(x,r)\coloneqq N_r\Bigparen{\vec s_i\bigparen{C_w^{(i-1)}(x,r)}}
 \]
 where $N_r$ is the neighbourhood of radius $r$ with respect to the distance $\dist$. Finally, let $C_w(x,r)\coloneqq C_w^{(\abs{w})}\! (x,r)$.
 
 Note that as $r$ tends to zero, the set $C_w^{(i)}(x,r)$ converges to the single point $\vec s_i\cdots\vec s_1(x)$ and in particular $C_w(x,r)$ converges to $w_{\rm rev}(x)$. By compactness, if $w_{\rm rev}$ does not have fixed points there exists a radius $r_w>0$ so that $x\notin C_w(x,r)$ for every $x\in X$ and $r\leq r_w$. We let
 \[
  t_0\coloneqq \max\left\{\frac{4\theta}{r_w}\ \middle|\ \abs{w}\leq 4\theta,\ \fix(w_{\rm rev})=\emptyset\right\}\cup\{1\}.
 \]
 
 Fix $t\geq t_0$. By Theorem~\ref{thm:discrete.fund.group_iso_quotient.of.jump.fund.grp}, the kernel of the discretisation map is $\angles {\rm FT_\theta}$; it is therefore enough to show that if $\Phi_S([\gamma],w)\in {\rm FT_\theta}$ then $w$ is conjugate to an element of $\rm Ell_\theta$. 
 
 Let $([\gamma],w)$ be such a pair and let $\vec{\bm\gamma}=\gamma\psi_S(w)=\gamma\alpha_w\jto[w^{-1}_{\rm rev}]x_0$. By hypothesis, there exists a \jp $\vec{\bm\gamma}'\in\CT_\theta$ which is freely\=/homotopic to $\vec{\bm\gamma}$. Tracing the base point under the free homotopy, we thus obtain a \jp $\vec{\bm \beta}$ so that $\vec{\bm\gamma}$ is homotopic to $\vec{\bm\beta}\vec{\bm\gamma}'\vec{\bm\beta}^*$. 
Choose a continuous path $\xi$ going from $x_0$ to $\vec{\bm\gamma}'(0)$ and let $([\gamma'],w')=\Phi_S^{-1}\bigparen{\xi\vec{\bm\gamma}'\xi^*}$, where we require that the word $w'$ matches exactly the sequence of (inverses of) jumps in the path $\vec{\bm\gamma}'$. Then $\abs{w'}\leq 4\theta$ and $w'$ is conjugated to $w$, as $([\gamma],w)$ is conjugated to $([\gamma'],w')$ by $\Phi_S^{-1}(\vec{\bm\beta}\xi^*)$. To prove our claim it is hence enough to show that $w'$ is in ${\rm Ell}_\theta$ \emph{i.e.} that it has a fixed point in $X$.
 
 Let $\vec{\bm\gamma}'=\gamma_0\jto[\vec s_1]\cdots\jto[\vec s_n]\gamma_n$ (so that $w'=\vec s_1^{\,-1}\cdots\vec s_n^{\,-1}$) and let $y\coloneqq\gamma_0(0)$. It is easy to show by induction that the image of $\gamma_i$ is contained in $C_{w'^{-1}_{\rm rev}}^{(i)}\bigparen{y,4\theta/t}$ (see Figure~\ref{fig:jpath.in.Cw}). If $w'$ (and hence $w'^{-1}$) did not have fixed points, by construction we would have $y\notin C_{w'^{-1}_{\rm rev}}\bigparen{y,4\theta/t}$ because $4\theta/t\leq r_{w'^{-1}_{\rm rev}}$ by definition. Still, $\vec{\bm\gamma}'$ is closed and therefore $\gamma_n(1)=y$ is in $C_{w'^{-1}_{\rm rev}}\bigparen{y,4\theta/t}$, a contradiction.
 
 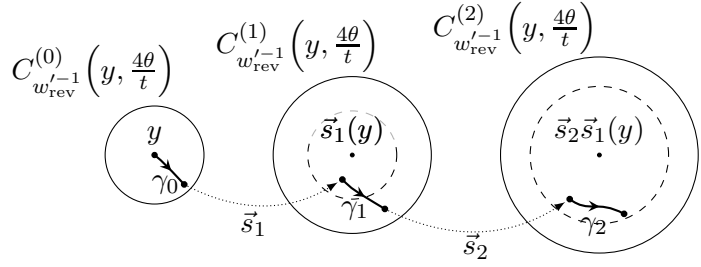
\begin{figure}
  \centering
  \begin{tikzpicture}[scale=1.3]
  \draw[dashed] (2,0) circle (0.45)node[above, fill=white, opacity=0.7]{$\vec s_1(y)$} node[above]{$\vec s_1(y)$} 
      (4.5,0) circle (0.7)node[above]{$\vec s_2\vec s_1(y)$} ;
  \draw (0,0)  node[bnode,inner sep=0.7]{} circle (0.5) node[above]{$y$}
      (2,0)  node[bnode,inner sep=0.5]{} circle (0.8)
    (4.5,0)  node[bnode,inner sep=0.5]{} circle (1) ; 
  \draw[thick,middlearrow=stealth] (0,0) ..controls (0.1,-0.085) .. (0.3,-0.3)
	  node[pos=0.5,below]{$\gamma_0$}
	  node[bnode,inner sep=0.6]{};
  \draw[thick,middlearrow=stealth] (1.9,-0.25) node[bnode,inner sep=0.6]{} 
	  ..controls (2,-0.35) .. (2.33,-0.55)
	  node[pos=0.5,below]{$\gamma_1$}
	  node[bnode,inner sep=0.6]{};
  \draw[thick,middlearrow=stealth] (4.2,-0.45) node[bnode,inner sep=0.6]{} 
	  ..controls (4.4,-0.6) and (4.47,-0.45) .. (4.75,-0.6) 
	  node[pos=0.5,below]{$\gamma_2$}
	  node[bnode,inner sep=0.6]{};
  \draw [densely dotted, bend right,->,shorten >= 2 pt](0.3,-0.3) to (1.9,-0.25) ;
  \draw [densely dotted, bend right,->,shorten >= 2 pt](2.33,-0.55) to (4.2,-0.45) ;
  \draw (1,-0.7) node{$\vec s_1$}
      (3.25,-0.93) node{$\vec s_2$}
      (-0.3,0.7) node[yshift=1.2 ex,xshift=-2.5 ex]{$C_{w'^{-1}_{\rm rev}}^{(0)}\Bigparen{y,\frac{4\theta}{t}}$}
      (1.55,0.95) node[yshift=1.3 ex,xshift=-0.8 ex]{$C_{w'^{-1}_{\rm rev}}^{(1)}\Bigparen{y,\frac{4\theta}{t}}$}
      (3.9,1.1) node[yshift=1.5 ex,xshift=-2 ex]{$C_{w'^{-1}_{\rm rev}}^{(2)}\Bigparen{y,\frac{4\theta}{t}}$};
  \end{tikzpicture}
  \caption{The jumping-path $\vec{\bm\gamma}'$ is contained in the sets $C_{w'^{-1}_{\rm rev}}^{(i)}\Bigparen{y,\frac{4\theta}{t}}$, where $w'^{-1}_{\rm rev}$ is the word $\vec s_1\cdots\vec s_n$.} 
  \label{fig:jpath.in.Cw}
 \end{figure}

\

Assume now that $X$ is semi\=/locally simply connected. Since it is compact, there exists a constant $\epsilon>0$ small enough so that every path contained in a ball of radius $\epsilon$ is homotopic in $X$ to a constant path. Moreover, by compactness there exist $\epsilon'\geq\epsilon''>0$ so that:
\begin{itemize}
 \item for every $w\in{\rm Ell}_\theta$ and $z\in\fix(w)$ we have $w\bigparen{B(z,\epsilon')}\subseteq B(z,\epsilon)$;
 \item any two points in $B(x,\epsilon'')$ can be joined with a continuous path contained in $B(x,\epsilon')$ (recall that $X$ is locally path\=/connected). 
\end{itemize}
Finally, we can further enlarge $t_0$ so that if $w\in {\rm Ell}_\theta$ then for every $y\in X$ such that $y\in C_{w_{\rm rev}^{-1}}\bigparen{y,4\theta/t_0}$ there exists a $z\in\fix(w)$ so that 
\begin{equation}\label{eq:condition.on.Cw}
C_{w_{\rm rev}^{-1}}\Bigparen{y,\frac{4\theta}{t_0}}\subseteq B(z, \epsilon''). 
\end{equation}

 Let again $t\geq t_0$ and $\Phi_S([\gamma],w)\in{\rm FT_\theta}$
 and let $\vec{\bm \gamma}',[\gamma'],w',y$ and $\xi$ be as above; it will be enough to show that $([\gamma'],w')\in K_\theta$. By the previous argument, we know that $w'\in {\rm Ell}_\theta$, therefore we only need to understand the continuous part of $\vec{\bm \gamma}'$. Note also that we already know that $y$ belongs to $C_{w_{\rm rev}^{-1}}\bigparen{y,4\theta/t_0}$ because $\vec{\bm \gamma}$ is a closed \jp, and hence by \eqref{eq:condition.on.Cw} there exists a fixed point $z\in\fix(w')$ so that $y$ is in $B(z, \epsilon'')$.
 
 It follows from Lemma~\ref{lem:paths.and.jumps.commute} that $\vec{\bm\gamma}'$ is homotopic to $\gamma''\jto[w'^{-1}_{\rm rev}]y$, where $\gamma''$ is an appropriate continuous path that we can suppose to be completely contained in $w'^{-1}\bigparen{C_{w'^{-1}_{\rm rev}}\bigparen{y,4\theta/t}}$. Note that the latter is in turn contained in $w'^{-1}\bigparen{B(z,\epsilon'')}$ and hence in $B(z,\epsilon)$. 
 
By construction, there exists a continuous path $\eta$ joining $y$ to $z$ with image contained in $B(z, \epsilon')$. Then both $\eta,w'(\eta)$ and $\gamma''$ are contained in $B(z, \epsilon)$ and therefore the closed path $(\gamma'')^*\eta w'(\eta^*)$ is null\=/homotopic. Note now that by definition $\bigparen{[\xi\eta w'(\eta^*\xi^*)\alpha_{w'}^*],w'}$ is in $K_\theta$ (see Figure \ref{fig:picture.thm.7}). Since we have
 \begin{align*}
  \Phi_S\bigparen{[\gamma'],w']}
      &=\Bigbrack{\xi\gamma''\jto[w'^{-1}_{\rm rev}]\xi^*} \\
      &=\Bigbrack{\xi\gamma''(\gamma'')^*\eta w'(\eta^*)\jto[w'^{-1}_{\rm rev}]\xi^*} \\
      &=\Bigbrack{\xi\eta w'(\eta^*\xi^*)\jto[w'^{-1}_{\rm rev}]x_0}
	=\Phi_S\bigparen{[\xi\eta w'(\eta^*\xi^*)\alpha_{w'}^*],w'} ,
 \end{align*}
 we can conclude that $[\gamma']=[\alpha_{w'}w'(\xi\eta)\eta^*\xi^*]$ because $\Phi_S$ is injective.
 
 \begin{figure}
  \centering
  \begin{tikzpicture}[scale=1.2]
    \draw[thick, middlearrow=stealth](0,0)node[bnode,inner sep=0.7]{}node[left]{$x_0$} 
	    ..controls(0.5,0.2) and (1.5,-0.1).. 
	    (2,0)node[above]{$\alpha_{w'}$} ..controls(2.5,0.2) and (3.5,-0.1).. 
	    (4,0)node[bnode,inner sep=0.7]{}node[right]{$w'(x_0)$};
    \draw[thick, middlearrow=stealth](1.5,2.5) node[bnode,inner sep=0.7]{}node[left]{$y$}
	    ..controls(1.8,2.45) and (2.1,2.45).. 
	    (2.5,2.5)node[pos=0.52,below]{$\gamma''$}node[bnode,inner sep=0.7]{}node[right]{$w'(y)$};
    \draw[thick, middlearrow=stealth](0,0) ..controls(0.8,0.4) and (1.3,2.4).. 
	    (1.5,2.5)node[pos=0.5,left]{$\xi$};
    \draw[thick, middlearrow=stealth] (4,0)..controls  (3.2,0.4) and (2.7,2.4) .. 
	    (2.5,2.5)node[pos=0.5,right]{$w'(\xi)$};

    \draw [densely dotted, bend right,->,shorten >= 5 pt](0,0) to (4,0) ;
    \draw [densely dotted, bend left,->,shorten >= 5 pt](1.5,2.5) to (2.5,2.5);
    \draw (2,-0.8)node{$w'$};
    
    \draw[thick, middlearrow=stealth] (1.5,2.5) ..controls  (1.75,2.9) .. (2,3.2)
	  node[pos=0.7,left]{$\eta$} 
	  node[bnode,inner sep=0.7]{}node[above]{$z$};
    \draw[thick, middlearrow=stealth] (2.5,2.5) ..controls  (2.25,2.9) .. (2,3.2)
	  node[pos=0.7,right]{$w'(\eta)$} ;
  \end{tikzpicture}
  \caption{Jumping paths in the kernel differ from elements of $K_\theta$ by short null-homotopic paths}
  \label{fig:picture.thm.7}
 \end{figure}
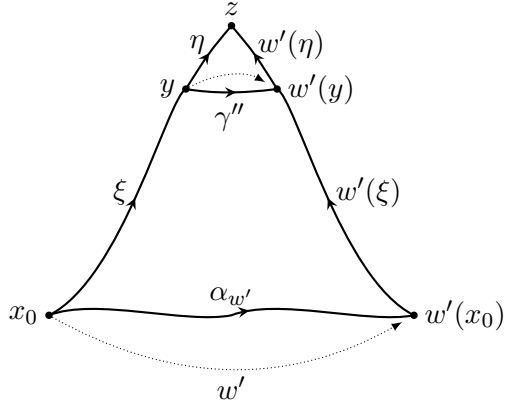

\end{proof}

Theorem \ref{thm:disc.fund.group.of.warped.cones} immediately allows us to compute the discrete fundamental groups of various warped cones. In particular we have the following:

\begin{cor}
 If $F_S$ acts by rotations on the sphere $\SS^2$, then $\thgrp\bigparen{(\SS^2,\wSdist[t]),x_0}=\{0\}$ for every $t$ and $\theta\geq 1$.
\end{cor} 

For level sets of warped cones, Proposition \ref{prop:discr.fund.group.is.semidirect.Gamma} yields a much sharper result. Indeed, let again $\Gamma=\angles{S\mid R}$ be a presentation of a (not necessarily finitely presented) finitely generated group, and let $\Gamma_\theta$ be the finitely presented group $\angles{S\mid R_\theta}$ where $R_\theta$ is the subset of $\aangles{R}$ of words of length at most $4\theta$. Then Theorem~\ref{thm:disc.fund.group.of.warped.cones} implies the following:

\begin{cor}\label{cor:discr.fund.grp.of.free.Gamma.warped.cone}
 Let $X$ be compact and semi\=/locally simply connected and let $\Gamma\curvearrowright X$ be a free action of $\Gamma=\angles{S\mid R}$ so that $F_S\curvearrowright X$ is jumping\=/geodesic. Then for every $t$ large enough we have
 \[
  \thgrp\bigparen{(X,\wSdist[t]),x_0}\cong 
  \frac{\pi_1(X)\rtimes_{\phi_S}F_S}{\aangles{\bigbrace{([\alpha_r^*],r)\bigmid r\in R_\theta}} }.  
 \]
\end{cor}
\begin{proof}
 An element of $K_\theta$ is of the form $\bigparen{[\beta w(\beta^*)\alpha_w^*],w}$ with $w\in{\rm Ell}_\theta$, $\beta(0)=x_0$ and $\beta(1)\in\fix(w)$. Since the action is free, ${\rm Ell}_\theta$ is equal to $R_\theta$, and hence $w=r\in R_\theta$. Moreover, we have $\beta w(\beta^*)=\beta r(\beta^*)=\beta\beta^*\sim x_0$ and hence $\bigparen{[\beta w(\beta^*)\alpha_w^*],w}=\bigparen{[\alpha_r^*],r}$. The statement now follows from Theorem~\ref{thm:disc.fund.group.of.warped.cones}.
\end{proof}

The proof of Proposition~\ref{prop:discr.fund.group.is.semidirect.Gamma} immediately implies the following:

\begin{cor}\label{cor:discr.fund.grp.of.free.Gamma.warped.cone.exact.sequence}
 Under the hypotheses of Corollary~\ref{cor:discr.fund.grp.of.free.Gamma.warped.cone}, there is a short exact sequence
 \[
  1\to \ \aangles{[\alpha_r]\mid r\in R_\theta}_{\thgrp\paren{X,\wSdist}}   \to\ 
  \thgrp\bigparen{(X,\wSdist),x_0} 
  \ \to \  
  \bigparen{G_\theta\rtimes_{\bar\phi} \Gamma_\theta} \to 1 , 
 \]
 where $G_\theta$ is the quotient of $\pi_1(X,x_0)$ as defined in Proposition~\ref{prop:discr.fund.group.is.semidirect.Gamma}.
\end{cor}

\section{Limits of discrete fundamental groups as coarse invariants for warped systems}\label{sec:limit.of.fund.grps}
A rather peculiar consequence of Theorem \ref{thm:disc.fund.group.of.warped.cones} is that the discrete fundamental group does not depend meaningfully on the parameter $t$:

\begin{cor}
 For every fixed $\theta\geq 1$, all the level sets $(X,\wSdist[t])$ with $t$ large enough have the same $\theta$\=/fundamental group.
\end{cor}

Also in view of the above, once an action $F_S\curvearrowright (X,d)$ is fixed, we feel invited to deal with the whole family of metric spaces $(X,\wSdist[t])$ at the same time and study its asymptotic properties. For this reason we give the following:

\begin{de}\label{def:warped.system}
 Given an action by homeomorphisms on a metric space $F_S\curvearrowright (X,d)$, its \emph{warped system} ${\rm WSys}\bigparen{F_S\curvearrowright (X,d)}$ is the data of the family of metric spaces $\bigbrace{(X,\wSdist[t])\bigmid t\in [1,\infty)}$ together with the set of generating homeomorphisms $S$.
 
 We say that a warped system \emph{satisfies a property $P$ asymptotically} if there exists a parameter $t_0$ large enough so that $(X,\wSdist[t])$ satisfies $P$ for every $t\geq t_0$. We will drop the distance $d$ from the notation if it is not relevant or it is clear from the context.
\end{de}

Using this new terminology, Theorem \ref{thm:disc.fund.group.of.warped.cones} then implies that if $X$ is compact and semi\=/locally simply connected and ${\rm WSys}\bigparen{F_S\curvearrowright X}$ is asymptotically jumping\=/geodesic then its $\theta$\=/discrete fundamental group is asymptotically isomorphic to $\bigparen{\pi_1(X,x_0)\rtimes_{\phi_S}F_S}\big/\aangles{K_\theta}$. To simplify the notation, we will denote this group by $\thgrp\bigparen{F_S\curvearrowright X}$.

\begin{convention*}
 From now on, we will always assume that ${\rm WSys}\bigparen{F_S\curvearrowright X}$ is an asymptotically jumping\=/geodesic warped system over a compact, semi\=/locally simply connected space $X$ with homotopy rectifiable paths.
\end{convention*}

Recall that from (\ref{item:lem:change.parameter}) of Lemma \ref{lem:qi.induce.maps.on.disc.fund.grp} we know that the identity map on $X$ induces a surjection $\thgrp\bigparen{F_S\curvearrowright X}\to\thgrp[\theta']\bigparen{F_S\curvearrowright X}$ for every choice of $\theta<\theta'$. That is, the family of asymptotic $\theta$\=/discrete fundamental groups forms a direct system. We can hence take the direct limit and define
\[
 \thgrp[\infty]\bigparen{F_S\curvearrowright X}\coloneqq \varinjlim\thgrp[\theta]\bigparen{F_S\curvearrowright X}.
\]

\begin{rmk}
 Note that in the definition of $\thgrp[\infty]\bigparen{F_S\curvearrowright X}$ it is important that we are taking the limit of \emph{asymptotic} discrete fundamental groups for a \emph{family} of metric spaces. Indeed $\thgrp[\infty]\bigparen{F_S\curvearrowright X}$ is \emph{not} the direct limit of the discrete fundamental groups of a single metric space, as such limit would always be trivial (every fixed $\theta$\=/path will become trivial if looked at with respect to a very large parameter $\theta'$).
 
 The situation is very different when considering the inverse limit. In fact, for a fixed metric space the inverse limit $\varprojlim\thgrp[\theta]\bigparen{(X,d),x_0}$ for $\theta\to 0$ is generally non trivial and it is often isomorphic to the fundamental group $\pi_1(X,x_0)$. This is the subject of \cite{Vig18c}.
\end{rmk}

\begin{rmk}
 Despite being inspired by it, the definition of $\thgrp[\infty]$ is quite different from the definition of \emph{coarse homology} of \cite{BCW14}.
\end{rmk}

As a consequence of Theorem~\ref{thm:disc.fund.group.of.warped.cones}, we can easily prove the following:

\begin{thm}\label{thm:coarse.fund.group.of.warped.systems}
 The group $\thgrp[\infty]\bigparen{F_S\curvearrowright X}$ is isomorphic to $\bigparen{\pi_1(X,x_0)\rtimes_{\phi_S}F_S}\big/\aangles{K_\infty}$
 where 
 \[
  K_\infty\coloneqq\bigbrace{\bigparen{[\beta w(\beta^*)\alpha_w^*],w}\bigmid \fix(w)\neq\emptyset,\ \beta(0)=x_0,\ \beta(1)\in\fix(w)}.
 \]
\end{thm}

\begin{proof}
 It is a standard fact that, given a group $\Gamma$ and an inverse system of normal subgroups $\Gamma_i\lhd\Gamma$ (\emph{i.e.} such that $\Gamma_i\subset\Gamma_j$ for every $i<j$), the direct limit $\varinjlim\Gamma/\Gamma_i$ is isomorphic to $\Gamma/\Gamma_\infty$\textemdash where $\Gamma_\infty=\bigcup_i\Gamma_i$.

 Then, the proof of the theorem follows easily from Theorem~\ref{thm:disc.fund.group.of.warped.cones}, as $\aangles{K_\infty}=\bigcup_{\theta>1}\aangles{K_\theta}$. 
\end{proof}

The interest of Theorem~\ref{thm:coarse.fund.group.of.warped.systems} is that the group $\thgrp[\infty]\bigparen{F_S\curvearrowright X}$ can prove to be an interesting coarse invariant for ${\rm WSys}\bigparen{F_S\curvearrowright X}$. For this we need to give another definition:

\begin{de}
We say that the warped system ${\rm WSys}\bigparen{F_S\curvearrowright X}$ has \emph{stable discrete fundamental group} if there exists a $\theta$ large enough so that the natural surjection $\thgrp\bigparen{F_S\curvearrowright X}\to\thgrp[\infty]\bigparen{F_S\curvearrowright X}$ is an isomorphism. 
\end{de}

It is simple to prove the following:

\begin{lem}\label{lem:stable.iff.finitely.presented(free.action)}
 If ${\rm WSys}\bigparen{F_S\curvearrowright X}$ is induced by a free action of a finitely generated group $\Gamma=\angles{S\mid R}$, then it has stable discrete fundamental group if and only if $\Gamma$ is finitely presented.
\end{lem}
\begin{proof}
 As in Corollary \ref{cor:discr.fund.grp.of.free.Gamma.warped.cone}, we let $R_\theta$ be the (finite) subset of $\aangles{R}$ of words of length at most $4\theta$. Then, since the $\Gamma$\=/action is free and every $r\in\aangles{R}$ acts trivially, we have $K_\theta=\bigbrace{\bigparen{[\alpha_r^*],r}\bigmid r\in R_\theta}$ and $K_\infty=\bigbrace{\bigparen{[\alpha_r^*],r}\bigmid r\in \aangles{R}}$. In particular, if $\aangles{K_\theta}=\aangles{K_\infty}$ for some $\theta\in\NN$, then $\aangles{R}=\aangles{R_\theta}$ and hence $\Gamma$ is finitely presented.
 
 Vice versa, if $\Gamma$ is finitely presented then there exists a $\theta\in\NN$ such that $\aangles{R_\theta}=\aangles{R}$. Following the lines of the proof of Proposition~\ref{prop:discr.fund.group.is.semidirect.Gamma}, we claim that $K_\infty$ is contained in $\aangles{K_\theta}$. Indeed, if $([\alpha_r^*],r)$ is in $\aangles{K_\theta}$, then for every $w\in F_S$ we have
 \begin{align*}
  \paren{[\alpha_{wrw^{-1}}^*],wrw^{-1}}
    &=\bigparen{[w(r(\alpha_{w^{-1}}^*)\alpha_r^*)\alpha_{w}^*],wrw^{-1}} \\
    &=\bigparen{[\alpha_{w}w(\alpha_r^*)\alpha_{w}^*],wrw^{-1}} \\
    &=\bigparen{\phi_S(w)[\alpha_r^*],wrw^{-1}} \\
    &=(e,w)\paren{[\alpha_r^*],r}(e,w)^{-1}
 \end{align*}
and hence $\paren{[\alpha_{wrw^{-1}}^*],wrw^{-1}}$ is in $\aangles{K_\theta}$ as well. Moreover, given $([\alpha_{r_1}^*],r_1)$ and $([\alpha_{r_2}^*],r_2)$ in $\aangles{K_\theta}$ we have
 \begin{align*}
  \bigparen{[\alpha_{r_1r_2}^*],r_1r_2}
  &=\bigparen{[r_1(\alpha_{r_2}^*)\alpha_{r_1}^*],r_1r_2} \\
  &=\bigparen{[\alpha_{r_2}^*\alpha_{r_1}^*],r_1r_2} \\
  &=\bigparen{[\alpha_{r_1}^*\alpha_{r_1}r_1(\alpha_{r_2})^*\alpha_{r_1}^*],r_1r_2} \\
  &=\bigparen{[\alpha_{r_1}^*],r_1}\bigparen{[\alpha_{r_2}^*],r_2}
 \end{align*}
and the latter is in $\aangles{K_\theta}$. The claim easily follows.

\end{proof}

\begin{rmk}\label{rmk:non.fin.pres.acting}
 Lemma \ref{lem:stable.iff.finitely.presented(free.action)} is enough to show that many interesting examples of warped systems have stable discrete fundamental group. Still, it also gives us means for constructing warped system with unstable discrete fundamental group. For example, it is well\=/known that there exists a finite set $S\subset F_2\times F_2$ so that the generated subgroup $\Gamma=\angles{S}<F_2\times F_2$ is not finitely presented (see \emph{e.g.}\cite[Section III.$\Gamma$.5, Example 5.4]{BrHa13}).
 Consider now any embedding of $F_2\times F_2$ in a compact Lie group $G$ (\emph{e.g.} an embedding in $G=\SO(3,\RR)\times \SO(3,\RR)$). 
This induces an embedding $\Gamma\hookrightarrow G$ and the isometric action by left translation produces a warped system ${\rm WSys}\bigparen{\Gamma\curvearrowright G}={\rm WSys}\bigparen{F_S\curvearrowright G}$ that, by Lemma~\ref{lem:stable.iff.finitely.presented(free.action)}, does not have stable discrete fundamental group.
\end{rmk}

\begin{lem}\label{lem:coarse.fund.group.invariant.of.sequences}
 Let ${\rm WSys}\bigparen{F_S\curvearrowright X}$ be a warped system with stable discrete fundamental group and let $(Y_k)_{k\in\NN}$ be a sequence of $1$\=/geodesic metric spaces. If there exists an increasing unbounded sequence $(t_k)_{k\in\NN}$ so that the family $\bigparen{X,\wSdist[t_k]}_{k\in\NN}$ is coarsely equivalent to $(Y_k)_{k\in\NN}$, then for every $\theta$ large enough there exists an $n\in\NN$ such that 
 \[
\thgrp(Y_k)\cong\thgrp[\infty]\bigparen{F_S\curvearrowright X}   
 \]
 for every $k\geq n$.
\end{lem}

\begin{proof}
 Since $\bigparen{F_S\curvearrowright X}$ has stable fundamental group, there exists a $\bar\theta$ large enough so that $\thgrp[\theta]\bigparen{F_S\curvearrowright X}\cong\thgrp[\infty]\bigparen{F_S\curvearrowright X}$ for every $\theta\geq \bar\theta$. In particular, from Theorem~\ref{thm:disc.fund.group.of.warped.cones} it follows that for every $\theta\geq\bar\theta$ there exists $n(\theta)$ so that 
 \[
  \thgrp[\theta]\bigparen{X,\wSdist[t_k]}\cong\thgrp[\infty]\bigparen{F_S\curvearrowright X}
 \]
 for every $k\geq n(\theta)$.
 
 Now, let $L$ and $A$ be the uniform constants of the quasi\=/isometries. Then for every $k\in\NN$ and $\theta\geq L\bar\theta+A$, by Lemma \ref{lem:qi.induce.maps.on.disc.fund.grp} we have a concatenation of surjections:
 \[
  \begin{tikzcd}
  \thgrp[\bar\theta]\bigparen{X,\wSdist[t_k]} \arrow[two heads,r] 
  & \thgrp[\theta]\bigparen{Y_k} \arrow[two heads,r]
  & \thgrp[L\theta+A]\bigparen{X,\wSdist[t_k]}
  \end{tikzcd}.
 \]
 
 If $k\geq n(L\theta+A)$, it follows from the discussion above that in the following diagram the maps are isomorphisms
 \[
  \begin{tikzcd}[column sep=0 ex]
  \thgrp[\bar\theta]\bigparen{X,\wSdist[t_k]}
	\arrow{rr}{\quad \id_*\quad}\arrow[swap]{rd}{\cong}
  & & \thgrp[L\theta+A]\bigparen{X,\wSdist[t_k]} \arrow{ld}{\cong}\\
  & \thgrp[\infty]\bigparen{F_S\curvearrowright X} &
  \end{tikzcd}
 \]
 and in particular $\id_*$ is an isomorphism. Hence by (\ref{item:lem:isomorphism}) of Lemma~\ref{lem:qi.induce.maps.on.disc.fund.grp} we deduce that
 \[
 \thgrp(Y_k)\cong\thgrp[\infty]\bigparen{F_S\curvearrowright X}.
 \]
\end{proof}

The above result can be further specialised in the study of coarse equivalences of warped systems:

\begin{thm}\label{thm:coarse.fund.group.invariant.of.stable.ws}
 Let ${\rm WSys}\bigparen{F_S\curvearrowright X}$ and ${\rm WSys}\bigparen{F_T\curvearrowright Y}$ be two warped systems. If there exist increasing unbounded sequences $(t_k)_{k\in\NN}$ and $(s_k)_{k\in\NN}$ so that the families $\bigparen{X,\wSdist[t_k]}_{k\in\NN}$ and $\bigparen{X,\wTdist[s_k]}_{k\in\NN}$ are coarsely equivalent and ${\rm WSys}\bigparen{F_S\curvearrowright X}$ has stable discrete fundamental group, then also ${\rm WSys}\bigparen{F_T\curvearrowright Y}$ has stable discrete fundamental group and
 \[
\thgrp[\infty]\bigparen{F_S\curvearrowright X}\cong\thgrp[\infty]\bigparen{F_T\curvearrowright Y}.  
 \]
\end{thm}
\begin{proof}
 Let $L$ and $A$ be the quasi\=/isometry constants of the coarse equivalence and fix three parameters $\theta,\theta'$ and $\theta''$ satisfying $\theta'\geq L\theta+A$ and $\theta''\geq L(L\theta'+A)+A$. For every $k\in\NN$, the quasi\=/isometries induce a concatenation of surjections
 \[
  \begin{tikzcd}[column sep=-2 em, row sep=2.5 em]
  \thgrp[\theta]\bigparen{X,\wSdist[t_k]}\arrow[two heads,rd]  
  && \thgrp[L\theta'+A]\bigparen{X,\wSdist[t_k]}\arrow[two heads,rd]
  && \thgrp[L\theta''+A]\bigparen{X,\wSdist[t_k]} \\
  &\thgrp[\theta']\bigparen{Y,\wTdist[s_k]} \arrow[two heads,ru]
  && \thgrp[\theta'']\bigparen{Y,\wTdist[s_k]} \arrow[two heads,ur]
  &
  \end{tikzcd}.
 \]
 If $\theta$ is large enough so that the projection $\thgrp\bigparen{F_S\curvearrowright X}\to\thgrp[\infty]\bigparen{F_S\curvearrowright X}$ is an isomorphism, then we can argue as in the proof of Lemma \ref{lem:coarse.fund.group.invariant.of.sequences} to deduce from Lemma~\ref{lem:qi.induce.maps.on.disc.fund.grp} and Theorem~\ref{thm:disc.fund.group.of.warped.cones} that for every $k$ large enough all the surjections above are actually isomorphisms.
 
 Since the composition map
 \[
  \begin{tikzcd}[column sep=-2 em, row sep=2.5 em]
  & \thgrp[L\theta'+A]\bigparen{X,\wSdist[t_k]}\arrow{rd}{\cong} &  \\
  \thgrp[\theta']\bigparen{Y,\wTdist[s_k]} \arrow{ru}{\cong} \arrow[dashed, rr] 
  && \thgrp[\theta'']\bigparen{Y,\wTdist[s_k]}
  \end{tikzcd}
 \]
 is induced by a map that is $A$\=/close to the identity, we deduce that for every $k$ large enough $(\id_Y)_*\colon\thgrp[\theta']\bigparen{Y,\wTdist[s_k]}\to\thgrp[\theta'']\bigparen{Y,\wTdist[s_k]}$ is an isomorphism. From this it follows that ${\rm WSys}\bigparen{F_T\curvearrowright Y}$ also has stable discrete fundamental group. 
 
 Now, the fact that $\thgrp[\infty]\bigparen{F_S\curvearrowright X}$ is isomorphic to $\thgrp[\infty]\bigparen{F_T\curvearrowright Y}$ follows trivially from Lemma~\ref{lem:coarse.fund.group.invariant.of.sequences}.
\end{proof}

\begin{cor}
 Let $\Gamma=\angles{S\mid R}$ be a finitely presented group and $\Lambda=\angles{T}$ be finitely generated. If there are free actions $\Gamma\curvearrowright X$ and $\Lambda\curvearrowright Y$ where $\pi_1(X)=\pi_1(Y)=\{0\}$ so that the induced warped systems ${\rm WSys}\bigparen{F_S\curvearrowright X}$ and ${\rm WSys}\bigparen{F_T\curvearrowright Y}$ admit coarsely equivalent unbounded subsequences, then $\Lambda$ is also finitely presented and $\Gamma\cong\Lambda$.
\end{cor}
\begin{proof}
 By Lemma~\ref{lem:stable.iff.finitely.presented(free.action)} the warped system ${\rm WSys}\bigparen{F_S\curvearrowright X}$ has stable discrete fundamental group and hence by Theorem~\ref{thm:coarse.fund.group.invariant.of.stable.ws} the same is true for ${\rm WSys}\bigparen{F_T\curvearrowright Y}$ and $\thgrp[\infty]\bigparen{F_S\curvearrowright X}$ is isomorphic to $\thgrp[\infty]\bigparen{F_T\curvearrowright Y}$. Again by Lemma~\ref{lem:stable.iff.finitely.presented(free.action)} we deduce that $\Lambda$ is finitely presented, and by Corollary~\ref{cor:discr.fund.grp.of.free.Gamma.warped.cone} we deduce that $\Gamma\cong\Lambda$.
\end{proof}

\begin{rmk}
 It is not clear to the author whether the group $\thgrp[\infty]\bigparen{F_S\curvearrowright X}$ is a coarse invariant of warped systems that do not have stable discrete fundamental group.
\end{rmk}

\section{Warped cones and box spaces}\label{sec:wc.and.box}

Also in this section we have the standing assumption that warped systems are locally jumping\=/geodesic and come from actions on compact semi\=/locally simply connected metric spaces with homotopy rectifiable paths. Let $\Lambda=\angles{T\mid R}$ be a (not necessarily finite) presentation of a finitely generated infinite group and, as before, let $R_\theta$ be the subset of $\aangles{R}$ of words of length at most $4\theta$ and let $\Lambda_\theta$ be the finitely presented group $\angles{T\mid R_\theta}$. Note that we have a natural surjection $\Lambda_\theta\to\Lambda$.

The next result is basically a rewriting of \cite[Lemma 3.4]{DeKh18}. We include a sketch of a proof for the convenience of the reader, as the statement we give is slightly more precise and general than what is proved by Delabie\textendash Khukhro.

\begin{thm}[\cite{DeKh18}]\label{thm:DeKh}
 Given $\theta\in\NN$ and a normal subgroup $N\lhd\Lambda$ such that the word length $\abs{g}$ is at least $4\theta$ for every $g\in N$, the $\theta$\=/discrete fundamental group of the Cayley graph of the quotient $\Lambda/N$ is given by
 \[
  \thgrp\bigparen{{\rm Cay}(\Lambda/N,T)}\cong N_\theta,
 \]
 where $N_\theta$ is the preimage of $N$ in $\Lambda_\theta$.
 
 In particular, if $\Lambda$ is finitely presented and $N_{k+1}\lhd N_k$ is a residual filtration of $\Lambda$, for every $\theta\gg 0$ there exists an $n\in\NN$ large enough so that $\thgrp\bigparen{{\rm Cay}(\Lambda/N_k,T)}\cong N_k$ for every $k\geq n$.
\end{thm}
\begin{proof}[Sketch of proof]
 The Cayley graph ${\rm Cay}(\Lambda/N,T)$ is isometric to the Cayley graph ${\rm Cay}(\Lambda_\theta/N_\theta,T)$. This can be made into a geodesic metric space by gluing in an interval $[0,1]$ for every edge of the graph. The set of homotopy classes of closed paths in ${\rm Cay}(\Lambda_\theta/N_\theta,T)$ is in bijective correspondence with the set of (reduced) words in $F_T$ that represent elements in $p^{-1}(N_\theta)\subseteq F_T$ (here $p$ denotes the surjection $F_T\to\Lambda_\theta$).
 
 We can apply Theorem~\ref{thm:discrete.fund.group_iso_quotient.of.jump.fund.grp} to the space ${\rm Cay}(\Lambda_\theta/N_\theta,T)$ to deduce that
 \[
  \thgrp\bigparen{ {\rm Cay}(\Lambda_\theta/N_\theta,T)}\cong 
  \frac{\pi_1\bigparen{{\rm Cay}(\Lambda_\theta/N_\theta,T)} }{ \angles{\rm FT_\theta} }
  \cong\frac{p^{-1}(N_\theta) }{ \angles{\rm FT_\theta}}
 \]
 (to apply Theorem~\ref{thm:discrete.fund.group_iso_quotient.of.jump.fund.grp} to metric spaces it is enough to consider the trivial action of the trivial group, so that the contributions of the `jumps' in the jumping\=/fundamental group is trivial).  
 
 By definition, $R_\theta$ is contained in $\rm FT_\theta$. Since $p^{-1}(N_\theta)/\aangles{R_\theta}=N_\theta\subseteq\Lambda_\theta$, we have
 \[
  \thgrp\bigparen{ {\rm Cay}(\Lambda_\theta/N_\theta,T)}
  \cong \frac{ N_\theta }{\angles{ {\rm FT}_\theta\cap N_\theta} }
 \]
 and the latter is simply equal to $N_\theta$ as the intersection $\rm FT_\theta\cap N_\theta$ is trivial by the assumption on the word length of the elements of $N$.
 
 If $\Lambda$ is finitely generated, there exists $\theta\gg 0 $ such that $\aangles{R_\theta}=\aangles{R}$. Moreover if $N_k$ is a residual sequence then all the elements of $N_k$ will have word length at least $4\theta$ for every $k$ large enough. The statement now easily follows from the above.
\end{proof}

We can now prove the following:

\begin{thm}\label{thm:box.space.qi.warpcone.constant.groups}
 Let $\Lambda=\angles{S\mid R}$ be an infinite finitely generated group. If a box space $\Box_{(N_k)}\Lambda$ of is coarsely equivalent to a subsequence $\bigparen{X,\wSdist[t_k]}_{k\in\NN}$ of a warped system ${\rm WSys}\bigparen{F_S\curvearrowright X}$, then ${\rm WSys}\bigparen{F_S\curvearrowright X}$ has stable discrete fundamental group if and only if $\Lambda$ is finitely presented. 
 
 Moreover, if $\Lambda$ is finitely presented then $N_k\cong \thgrp[\infty]\bigparen{F_S\curvearrowright X}$ for every $k$ large enough.
\end{thm}
\begin{proof}
 Note that since $\Lambda$ is infinite the sequence $t_k$ must be unbounded. Now the proof follows closely the proof of Theorem~\ref{thm:coarse.fund.group.invariant.of.stable.ws}: assume that ${\rm WSys}\bigparen{F_S\curvearrowright X}$ has stable discrete fundamental group, let $L$ and $A$ be the quasi\=/isometry constants of the coarse equivalence and let $\theta,\theta',\ \theta''$ satisfy $\theta'\geq L\theta+A$ and $\theta''\geq L(L\theta'+A)+A$. For every $k\in\NN$, the quasi\=/isometries induce a concatenation of surjections
 \[
  \begin{tikzcd}[column sep=-3 em, row sep=2.5 em]
  \thgrp[\theta]\bigparen{X,\wSdist[t_k]}\arrow[two heads,rd]  
  && \thgrp[L\theta'+A]\bigparen{X,\wSdist[t_k]}\arrow[two heads,rd]
  && \thgrp[L\theta''+A]\bigparen{X,\wSdist[t_k]} \\
  &\thgrp[\theta']\bigparen{{\rm Cay}(\Lambda/N_k,T)} \arrow[two heads,ru]
  && \thgrp[\theta'']\bigparen{{\rm Cay}(\Lambda/N_k,T)} \arrow[two heads,ur]
  &
  \end{tikzcd}
 \]
 and just as in Theorem~\ref{thm:coarse.fund.group.invariant.of.stable.ws} we can deduce that if $\theta$ is large enough so that the $\theta$\=/discrete fundamental group of ${\rm WSys}\bigparen{F_S\curvearrowright X}$ stabilised, then for every $k$ large enough $\id_*\colon\thgrp[\theta']\bigparen{{\rm Cay}(\Lambda/N_k,T)}\to\thgrp[\theta'']\bigparen{{\rm Cay}(\Lambda/N_k,T)}$ is an isomorphism. 
 
 From Theorem~\ref{thm:DeKh} we know that $\thgrp[\theta']\bigparen{{\rm Cay}(\Lambda/N_k,T)}\cong (N_k)_{\theta'}<\Lambda_{\theta'}$ and from its proof it also follows that the map $\id_*$ coincides with the quotient $(N_k)_{\theta'}\to(N_k)_{\theta''}$ induced from $\Lambda_{\theta'}\to\Lambda_{\theta''}$.
 
 Now, if $\aangles{R}$ was strictly larger than $\aangles{R_{\theta'}}$ we could choose a relation $r\in\aangles{R}\smallsetminus\aangles{R_{\theta'}}$. Choosing a $\theta''$ larger than $\abs{r}/4$, we would find that $r$ denotes an element in the kernel of $\id_*\colon (N_k)_{\theta'}\to (N_k)_{\theta''}$ for every $k$ large enough. Still, since the sequence $N_k$ is residual, there must be a $k$ large enough so that $r$ is not trivial in $(N_k)_{\theta'}$, and this contradicts the fact that $\id_*$ is an isomorphism. 
 
 Since $R_\theta$ is finite, we deduce that if ${\rm WSys}\bigparen{F_S\curvearrowright X}$ has stable discrete fundamental group then $\Lambda$ is finitely presented. The inverse implication is analogous. 
 
 The `moreover' part of the statement follows immediately from Theorem~\ref{thm:DeKh} and Lemma~\ref{lem:coarse.fund.group.invariant.of.sequences}.
\end{proof}

\begin{rmk}
 In the proof of Theorem~\ref{thm:box.space.qi.warpcone.constant.groups} we never used the fact that the residual sequence $N_k\lhd\Lambda$ consists of nested subgroups. Indeed, we only need that for every $\theta\in\NN$ there exists an $n$ large enough so that $N_k$ consists only of elements of length at least $4\theta$ for every $k\geq n$. 
\end{rmk}

Theorem~\ref{thm:box.space.qi.warpcone.constant.groups} implies that box spaces and warped systems tend to have very different coarse geometry. We wish to give some examples of such differences in the next two subsections. In what follows we will say (with an abuse of notation) that a box space is coarsely equivalent to a warped system if it is coarsely equivalent to an unbounded sequence of spaces in a warped system.

\subsection{Box spaces that are not coarsely-equivalent to warped systems} It follows from Theorem~\ref{thm:box.space.qi.warpcone.constant.groups} that for a box space of a finitely presented group $\Box_{(N_k)}\Lambda$ to be coarsely equivalent to a warped system it is necessary that the groups $N_k$ be all isomorphic for $k$ large enough. 

Note that if $\Lambda\cong F_n$ is a free group, then the rank of the normal subgroup $N_k$ is known to be ${\rm rk}(N_k)=(n-1)[F_n:N_k]+1$ and hence a bound on the rank of $N_k$ implies a bound on the index $[F_n:N_k]$. It follows that no box space of a free group can be coarsely equivalent to a warped system.

More in general, recall that the \emph{rank gradient} of a residual filtration is defined as 
\[
{\rm RG}(\Lambda,(N_k))\coloneqq \lim_{k\to\infty}\frac{{\rm rk}(N_k)}{[\Lambda:N_k]}.
\]
Recall also that if $\Lambda$ has \emph{fixed price} $p$ (for a definition and discussion see \emph{e.g.} \cite{Fur11}), then every residual filtration has rank gradient $p-1$.

\begin{cor}
 If $\Lambda$ is finitely presented and $\Box_{(N_k)}\Lambda$ is coarsely equivalent to a warped system then ${\rm RG}(\Lambda,(N_k))=0$. In particular, if $\Lambda$ has fixed price $p>1$, then no box space of $\Lambda$ is coarsely equivalent to a warped system.
\end{cor}

A quite different reason for box spaces to not be coarsely equivalent to warped systems goes as follows. Let $\Lambda$ be a lattice in a simple Lie group $G$ not locally isomorphic to $\Sl(2,\RR)$. Then the finite index subgroups $N_k$ are lattices as well and hence the Mostow Rigidity Theorem applies.
That is, if $N_k$ is isomorphic to $N_{k'}$ then $N_k$ and $N_k'$ are actually conjugated in $G$ and hence $G/N_k$ and $G/N_{k'}$ have the same (finite) volume with respect to the Haar measure. Still, $G/N_k$ is a cover of $G/\Lambda$ of rank $[\Lambda: N_k]$ and hence it has volume $\vol(G/N_k)=[\Lambda: N_k]\vol(G/\Lambda)$, which is again implying an upper bound on the index in terms of the isomorphism class of $N_k$. We hence proved the following:

\begin{cor}\label{cor:lattices.are.not.warped.systems}
 If $\Lambda$ is lattice in a simple Lie group $G$ not locally isomorphic to $\Sl(2,\RR)$, then no box space of $\Lambda$ is coarsely equivalent to a warped system.
\end{cor}

\subsection{Warped systems that are not coarsely-equivalent to box spaces}
We already noted that the warped system ${\rm WSys}\bigparen{F_2\curvearrowright \SS^2}$ induced by an action by rotations has stable discrete fundamental group and we have $\thgrp[\infty]\bigparen{F_2\curvearrowright \SS^2}=\{e\}$. It follows from Theorem~\ref{thm:box.space.qi.warpcone.constant.groups} that if ${\rm WSys}\bigparen{F_2\curvearrowright \SS^2}$ was coarsely equivalent to a box space then the quotienting groups $N_k$ should be trivial and hence $\Lambda$ would be finite, a contradiction. This is particularly interesting as it was shown in \cite{Vig18b} that such warped systems can be used to produce expander graphs.

The argument above relies on the observation that, in that specific case, any group having  $\thgrp[\infty]\bigparen{F_S\curvearrowright X}$ as a finite index subgroup could not have box spaces coarsely equivalent to warped systems. This strategy can be applied in other cases as well:

\begin{cor}
 Assume that ${\rm WSys}\bigparen{F_S\curvearrowright X}$ has stable discrete fundamental group. If every group $\Lambda$ containing $\thgrp[\infty]\bigparen{F_S\curvearrowright X}$ as a finite index subgroup does not admit box spaces that are coarsely equivalent to a warped system, then ${\rm WSys}\bigparen{F_S\curvearrowright X}$ is not coarsely equivalent to any box space. 
 
 In particular, this is the case when $\thgrp[\infty]\bigparen{F_S\curvearrowright X}$ is:
 \begin{enumerate}[\rm (a)]
  \item a finite group;
  \item a non\=/residually finite group;
  \item a non\=/finitely presented group;
  \item a lattice in a higher rank simple Lie group.
 \end{enumerate}
\end{cor}
\begin{proof}
 Case (a) is obvious and case (b) follows from the fact that we insist that box spaces be generated by residual sequences and therefore the group $\Lambda$ (and its subgroups) would have to be residually finite.
 
 Case (c) holds true as Theorem~\ref{thm:box.space.qi.warpcone.constant.groups} implies that the group $\Lambda$ should be finitely presented, and therefore the same should be true for its finite index subgroups.
 
 Case (d) follows from the fact that lattices in higher rank Lie groups are rigid under quasi\=/isometries \cite{KlLe97}. This implies that a group $\Lambda$ containing such a lattice as a finite index subgroup would have to be itself a lattice and hence Corollary~\ref{cor:lattices.are.not.warped.systems} would apply.
\end{proof}

\begin{rmk}
 By Corollary~\ref{cor:discr.fund.grp.of.free.Gamma.warped.cone.exact.sequence}, to find examples of warped systems with stable discrete fundamental group so that $\thgrp[\infty]\bigparen{F_S\curvearrowright X}$ is not residually finite it would be enough to find a free action of a finitely presented but not residually finite group.
 
 We feel that it should be possible to find examples of warped systems with stable fundamental group for which $\thgrp[\infty]\bigparen{F_S\curvearrowright X}$ is not finitely presented. Still, Lemma~\ref{lem:stable.iff.finitely.presented(free.action)} implies that we cannot hope to find such an example by considering free actions on `pleasant' compact spaces.
\end{rmk}

Since in the definition of residual sequence we insist that the subgroups be nested, we immediately have the following stronger result:

\begin{cor}\label{cor:cohopf.wc.non.box}
 Assume that ${\rm WSys}\bigparen{F_S\curvearrowright X}$ has stable discrete fundamental group. If $\thgrp[\infty]\bigparen{F_S\curvearrowright X}$ does not contain a finite index normal subgroup isomorphic to $\thgrp[\infty]\bigparen{F_S\curvearrowright X}$ itself (\emph{i.e.} it is \emph{co\=/Hopfian}), then ${\rm WSys}\bigparen{F_S\curvearrowright X}$ is not coarsely equivalent to any box space.
\end{cor}

We wish to remark here that \emph{many} groups are co\=/Hopfian, see \cite{vLi17} for an exhaustive study of such groups.

\

Note that if $\Gamma\curvearrowright X$ is an action of a finitely generated group and $S$ and $T$ are two finite sets of generators, then ${\rm WSys}\bigparen{F_S\curvearrowright X}$ and ${\rm WSys}\bigparen{F_T\curvearrowright X}$ are naturally coarsely equivalent. Moreover, using Theorem~\ref{thm:coarse.fund.group.of.warped.systems} it is simple to prove that $\thgrp[\infty]\bigparen{F_S\curvearrowright X}$ is naturally isomorphic to $\thgrp[\infty]\bigparen{F_T\curvearrowright X}$. In view of these facts, in what follows we feel justified to simply use the notation ${\rm WSys}\bigparen{\Gamma\curvearrowright X}$ and $\thgrp[\infty]\bigparen{\Gamma\curvearrowright X}$ to denote the (coarse equivalence class) of the warped system induced by $\Gamma\curvearrowright M$ and the limit of its discrete fundamental groups.

One of the main results of \cite{dLVi18} is that the warped system ${\rm WSys}\bigparen{\Gamma_d\curvearrowright \SO(d,\RR)}$ is not coarsely equivalent to a box space of a lattice in a higher rank semisimple algebraic group: this allowed us to prove that some new examples of superexpanders were not coarsely equivalent to previously known examples due to Lafforgue (here $\Gamma_d=\SO(d,\ZZ[\frac{1}{5}])$ and $d\geq 5$). We will now complete that result by showing that such a warped system is not coarsely equivalent to any box space.

The following is proved in \cite[Theorem 5.8]{dLVi18}:
\begin{thm}[\cite{dLVi18}]
 Let $\Gamma=\angles{S}$ be a finitely generated group and $\Gamma\curvearrowright M$ an essentially free action by isometries on a compact Riemannian manifold. If (a sequence in) ${\rm WSys}\bigparen{\Gamma\curvearrowright M}$ is coarsely equivalent to a box space $\Box_{N_k}\Lambda$ then
 $\Lambda$ is quasi\=/isometric to $\Gamma\times\ZZ^{\dim(M)}$.
\end{thm}

Assume now that $\Gamma$ is finitely presented, $M$ has finite fundamental group and that the action $\Gamma\curvearrowright M$ is free and by isometries. Then Corollary \ref{cor:discr.fund.grp.of.free.Gamma.warped.cone.exact.sequence} implies that $\thgrp[\infty]\bigparen{\Gamma\curvearrowright M}$ is virtually isomorphic to $\Gamma$ (recall that two groups are \emph{virtually isomorphic} if they are equivalent under the equivalence relation induced by taking quotients by finite subgroups or passing to finite index subgroups). If ${\rm WSys}\bigparen{\Gamma\curvearrowright M}$ is coarsely equivalent to a box space of $\Lambda$, it follows that $\Gamma$ is virtually isomorphic to $\Lambda$ as well and it is hence quasi\=/isometric to $\Gamma\times\ZZ^{\dim(M)}$, which is often not the case. For example, we immediately get the following:

\begin{cor}
 Let $\Gamma\curvearrowright M$ be a free action by isometries of a finitely presented group on a Riemannian manifold with finite fundamental group. If either $\Gamma$
 \begin{itemize}
  \item has polynomial growth;
  \item has property (T);
  \item is Gromov hyperbolic;
 \end{itemize}
 then ${\rm WSys}\bigparen{\Gamma\curvearrowright M}$ is not coarsely equivalent to any box space.
 
 In particular, superexpanders obtained from the warped system ${\rm WSys}\bigparen{\Gamma_d\curvearrowright \SO(d,\RR)}$ are not coarsely equivalent to any box space.
\end{cor}

We would like to remark that it is possible to prove the above statement about virtual isomorphisms directly from Theorem~\ref{thm:disc.fund.group.of.warped.cones} without passing through Corollary~\ref{cor:discr.fund.grp.of.free.Gamma.warped.cone.exact.sequence} (and hence avoiding Proposition~\ref{prop:discr.fund.group.is.semidirect.Gamma}). We wish to do so explicitly, as we think that this technique is interesting in its own right.

\begin{thm}\label{thm:virtual.isomorphism.through.coverings}
 Let $\Gamma\curvearrowright M$ be a free action of a finitely generated group on a compact manifold with finite fundamental group. Then $\thgrp[\infty]\bigparen{\Gamma\curvearrowright M}$ is virtually isomorphic to $\Gamma$.
\end{thm}
\begin{proof}
 Let $\angles{S\mid R}$ be a presentation of $\Gamma$ and consider the universal cover $\widetilde{M}\to M$. For every $s\in S$, choose a lift $\tilde s$ to the universal cover:
 \[
  \begin{tikzcd}
   \widetilde{M} \arrow{d}\arrow{r}{\tilde s} &\widetilde{M} \arrow{d}\\
   M \arrow{r}{s} & M
  \end{tikzcd}
 \]
this induces an action $\tilde\rho\colon F_S\to\isom\bigparen{\widetilde{M}}$ and an associated warped system ${\rm WSys}\bigparen{F_S\curvearrowright\widetilde{M}}$.

Note now that $\tilde\rho(R)$ is a subset of the group of deck transformations of $\widetilde{M}$, which is a finite group by hypothesis. It follows that $\ker(\tilde \rho)$ is a subgroup of finite index of $\aangles{R}\subset F_S$.

Let $\widetilde\Gamma\coloneqq F_S/\ker(\tilde\rho)$ and note that $\Gamma$ is the quotient of $\widetilde\Gamma$ by the finite subgroup $\aangles{R}/\ker(\tilde \rho)$. Since $\widetilde\Gamma\curvearrowright \widetilde M$ is a free action on a simply connected manifold, we can apply Corollary~\ref{cor:discr.fund.grp.of.free.Gamma.warped.cone} to the warped system ${\rm WSys}\bigparen{\widetilde\Gamma\curvearrowright \widetilde{M}}={\rm WSys}\bigparen{F_S\curvearrowright\widetilde{M}}$ to deduce that
\[
\thgrp\bigparen{F_S\curvearrowright \widetilde{M}}
\cong{\widetilde\Gamma}_\theta 
\]
where $\widetilde\Gamma_\theta$ is the group $F_S/\aangles{\{w\in \ker(\tilde\rho)\mid \abs{w}\leq4\theta\} }$. 

Note that at the level of jumping\=/fundamental groups it is simple to mimic the theory of topological covers and deduce that for every $t\in\RR_+$ there is an injection 
\[
\begin{tikzcd}
 \jpgrp\bigparen{t\cdot \widetilde M}\arrow[hook]{r} &\jpgrp\bigparen{t\cdot M} 
\end{tikzcd}
\]
whose image is a subgroup of index (at most) $\abs{\pi_1(M)}$ (this map coincides with the natural inclusion $F_S\hookrightarrow \pi_1(M)\rtimes F_S$). In the above we added $t$ to the notation to remember that we are working with metrics scaled by $t$.

Since the quotient map $\widetilde M\to M$ is $1$\=/Lipschitz with respect to the warped metrics $\wSdist[t]$, it follows from Theorem~\ref{thm:discrete.fund.group_iso_quotient.of.jump.fund.grp} that the above injection descends to a homomorphism between the discrete fundamental groups \emph{via} the discretisation procedure
\[
\begin{tikzcd}
 &\jpgrp\bigparen{t\cdot \widetilde M}\arrow[hook]{r}\arrow[swap,two heads]{d}{\discrmap} 
	&\jpgrp\bigparen{t\cdot M} \arrow[two heads]{d}{\discrmap} \\
 \widetilde\Gamma_\theta\arrow[two heads]{r}{}&\thgrp\Bigparen{\widetilde M,\wSdist[t]}\arrow[dashed]{r} 
	&\thgrp\Bigparen{ M,\wSdist[t]} 
\end{tikzcd} 
\]
and that the image has index $\abs{\pi_1(M)}$.

Since the above homomorphisms do not depend on $t$ (as long as $t$ is large enough), they induce a homomorphism of the direct systems as $\theta$ varies in $\NN$, and therefore induce a limit homomorphism
\[
 \widetilde\Gamma\cong \varinjlim\widetilde\Gamma_\theta
 \longrightarrow\varinjlim\thgrp\bigparen{ M,\wSdist[t]} =\thgrp[\infty]\bigparen{ M,\wSdist[t]} 
\]
whose image is a finite index subgroup. Moreover, using Theorem~\ref{thm:disc.fund.group.of.warped.cones} it is easy to check that this limit homomorphism is actually injective, thus completing the proof.
\end{proof}

\subsection{Warped systems that are coarsely equivalent to box spaces}
Despite all the examples provided above, warped systems over compact manifolds and box spaces \emph{can} be coarsely equivalent. The easiest example is probably the following: let $X=\TT^d$ be the $d$\=/dimensional torus and consider the trivial warped system ${\rm WSys}\bigparen{\{e\}\curvearrowright\TT^d}$. 
It is then easy to see that $(\TT^d,\wSdist[n])$ is just the torus with the metric rescaled by $n$ and it is hence quasi\=/isometric to the finite quotient $\bigparen{\ZZ/n\ZZ}^d\cong \ZZ^d/(n\ZZ)^d$. That is, ${\rm WSys}\bigparen{\{e\}\curvearrowright\TT^d}$ is coarsely equivalent to a box space of $\ZZ^d$.

The above example can be made quite more interesting using a result of Kielak and Sawicki. In \cite[Appendix]{Saw17} they show that there exist (uncountably many) actions $\ZZ^k\curvearrowright\TT^d$ by rotations such that ${\rm WSys}\bigparen{\ZZ^k\curvearrowright\TT^d}$ is coarsely equivalent to ${\rm WSys}\bigparen{\{e\}\curvearrowright\TT^{d+k}}$ and it is hence coarsely equivalent to a box space of $\ZZ^{d+k}$. 
%

\

For the next example we will need to consider a sequence of non-normal finite index subgroups $N_k<\Lambda$. In this case the Cayley graphs of the quotients are not defined and we thus need to use the Schreier coset graphs ${\rm Schr}(\Lambda/N_k)$. Note also that the sequence of subgroups we use is not residual. In particular, the following example is technically not a box space (according to our definition), but we still find it interesting:

\begin{exmp}\label{exmp:wc.qi.box.expander}
 Let $\Lambda\coloneqq \ZZ^2\rtimes\Sl(2,\ZZ)$ where $\Sl(2,\ZZ)\curvearrowright \ZZ^2$ is the natural action. Note that $k\ZZ^2$ is a characteristic subgroup of $\ZZ^2$ and hence $N_k\coloneqq (k\ZZ)^2\rtimes\Sl(2,\ZZ)$ is a subgroup of $\Lambda$. Moreover, it is simple to show that $N_k\cong\Lambda$ for every $k$, so that the (non\=/normal) box space $\Box_{N_k}\Lambda$ could be coarsely equivalent to some warped system, and this is actually the case. 
Indeed, consider the natural action $\Sl(2,\ZZ^2)\curvearrowright\TT^2$. It is then a relatively simple task to check that the spaces $(\TT^2,\wSdist[n])$ and ${\rm Schr}(\Lambda/N_k)$ are uniformly quasi\=/isometric. 
 
 The interest of this example is that the Schreier graphs ${\rm Schr}(\Lambda/N_k)$ form a family of expanders (this is actually a celebrated early example of expanders provided by Margulis in \cite{Mar73}, see also \cite{GaGa81}). 
 In particular, the warped system ${\rm WSys}\bigparen{\Sl(2,\ZZ)\curvearrowright\TT^2}$ is as far as possible from a nicely behaved warped system such as ${\rm WSys}\bigparen{\{e\}\curvearrowright\TT^d}$.
\end{exmp}

\begin{rmk}
 As currently stated, Theorem~\ref{thm:DeKh} is not generally true for Schreier graphs ${\rm Schr}(\Lambda/N)$ where the subgroup $N$ is not normal in $\Lambda$. The proof we provided fails only at its very last step \emph{i.e.} it is no longer true that ${\rm FT}_\theta\cap N_\theta$ is trivial. It can still be useful to characterise $\thgrp\bigparen{ {\rm Schr}(\Lambda_\theta/N_\theta)}$ as $\frac{ N_\theta }{\angles{ {\rm FT}_\theta\cap N_\theta} }$ though.
\end{rmk}

\appendix

\section{Discrete fundamental group for (unified) warped cones}\label{sec:appendix}

In this appendix we wish to show that much of the work here developed for warped systems can be adapted to warped cones as they were originally defined in \cite{Roe05}. We begin by properly recalling the definition: given a warped system ${\rm WSys}(S\curvearrowright X)$, the action of $F_S$ trivially extends to a level\=/preserving action on the direct product $X\times [1,\infty)$. The space $X\times [1,\infty)$ can be equipped with a `conical' metric $d_{\rm cone}$ in various (roughly equivalent) ways. For example, if $(X,g)$ is a Riemannian manifold it is natural to define $d_{\rm cone}$ by $t\cdot g+dt^2$ (this is Roe's original definition). For more general spaces one can do as in \cite{Saw18} or use the $0$\=/cone metric as in \cite[Chapter I.5]{BrHa13}. The
\emph{(unified) warped cone} $\cone_S (X)$ is the space $(X\times[1,\infty),\wSdist)$ where $\wSdist$ is the warped metric obtained warping $d_{\rm cone}$ with the $F_S$\=/action. 

As in the introduction, we denote by $\cone_S^t(X)$ the subset $X\times\{t\}\subset\cone_S(X)$ equipped with the induced metric. These are called \emph{level sets} of the warped cone. Further, we will denote by $\cone_S^{[a,b]}(X)$ the subset $X\times[a,b]\subseteq\cone_S(X)$ with the induced metric.

Given any sensible choice of the conical metric $d_{\rm cone}$, the level sets $\cone_S^t(X)$ and the spaces $(X,\wSdist[t])$ are uniformly quasi\=/isometric (see \emph{e.g.} \cite{Vig18b}). For this reason we tend to confound them and call both of them `level sets'. 

\

As in the the last few sections, we still assume the space $X$ to be a `nice' compact space and ${\rm WSys}\bigparen{S\curvearrowright X}$ to be jumping\=/geodesic.
For any fixed $\theta\geq 1$, it is easy to show that for $t\gg 0$ large enough $\thgrp\bigparen{\cone_S^t(X)}\cong\thgrp\bigparen{X,\wSdist[t]}$. Moreover it is also simple to prove the following lemma:

\begin{lem}\label{lem:appendix}
 For every $\theta\geq 1$ there exists a $t_0$ large enough so that for every  $t_0\leq a \leq t\leq b\leq\infty$ the natural inclusion and projection
\[
  \begin{tikzcd}
  \cone_S^{t}(X)\arrow[hook]{r}{\iota}
  & \cone_S^{[a,b]}(X)\arrow[two heads]{r}{p}
  & \cone_S^{t}(X)
  \end{tikzcd}
\]
induce isomorphisms
\[
  \begin{tikzcd}
  \thgrp\bigparen{\cone_S^{t}(X)}\arrow[]{r}{\iota_*}
  & \thgrp\bigparen{\cone_S^{[a,b]}(X)}\arrow[]{r}{p_*}
  & \thgrp\bigparen{\cone_S^{t}(X)}
  \end{tikzcd}.
\]
\end{lem}

This Lemma allows us to mimic the proof of Theorem~\ref{thm:coarse.fund.group.invariant.of.stable.ws} in the context of (unified) warped cones.

\begin{thm}
 If ${\rm WSys}\bigparen{F_S\curvearrowright X}$ has stable discrete fundamental group and $\cone_S(X)$ is quasi\=/isometric to $\cone_T(Y)$ then ${\rm WSys}\bigparen{F_T\curvearrowright Y}$ has stable discrete fundamental group and $\thgrp[\infty]\bigparen{F_S\curvearrowright X}\cong\thgrp[\infty]\bigparen{F_T\curvearrowright Y}$.
\end{thm}

\begin{proof}[Sketch of proof]
  Let $f\colon\cone_S(X)\to\cone_T(Y)$ be an $(L,A)$\=/quasi\=/isometry and let $\bar f$ be the coarse inverse. Also, fix three parameters $\theta,\theta'$ and $\theta''$ satisfying $\theta\geq L+A$, $\theta'\geq L\theta+A$ and $\theta''\geq L(L\theta'+A)+A$ with $\theta$ large enough so that the projection $\thgrp\bigparen{F_S\curvearrowright X}\to\thgrp[\infty]\bigparen{F_S\curvearrowright X}$ is an isomorphism. 
  
  For every $a\gg 1$ there exist $c,b\gg 1$ such that
 \[
f\Bigparen{\cone_S^{[c,\infty]}(X)}\subseteq \cone_T^{[b,\infty]}(Y) \quad\text{ and }\quad 
\bar f\Bigparen{\cone_T^{[b,\infty]}(Y)}\subseteq \cone_S^{[a,\infty]}(X).  
 \]
  By Lemma \ref{lem:appendix}, we can deduce that both 
  \[
   f_*\colon \thgrp\bigparen{\cone_S^{[c,\infty]}(X)}
    \longrightarrow \thgrp[\theta']\bigparen{\cone_T^{[b,\infty]}(Y)}
  \]
  and
  \[
   \bar f_*\colon \thgrp[\theta']\bigparen{\cone_T^{[b,\infty]}(Y)}
    \longrightarrow \thgrp[L\theta'+A]\bigparen{\cone_S^{[a,\infty]}(X)}
  \]
 are surjective. Indeed, every $\theta'$\=/path $Z$ in $\cone_T^{[b,\infty]}(Y)$ is equivalent to a $1$\=/path in a level set which is sufficiently high up so that its image under $\bar f$ is a $\theta$\=/path in $\cone_S^{[c,\infty]}(X)$. This $\theta$\=/path is then is mapped to $[Z]$ by $f_*$. The same argument works for $\bar f_*$ as well.
  
  We can now find parameters $a>a'>a''\gg 1$ and $b>b'\gg 1$ so that the following composition of maps makes sense and it induces a commutative diagram:
 \[
  \begin{tikzcd}[column sep=-3 em, row sep=2.5 em]
  \thgrp[\theta]\bigparen{\cone_S^{a}(X)}\arrow[]{d}{\iota_*}\arrow[dashed]{rr}{}
  && \thgrp[L\theta'+A]\bigparen{\cone_S^{a}(X)}
	\arrow[shift left]{d}{\iota_*}\arrow[dashed]{rr}
  && \thgrp[L\theta''+A]\bigparen{\cone_S^{a}(X)} \\
  \thgrp[\theta]\bigparen{\cone_S^{[a,\infty]}(X)}\arrow[two heads]{rd}{f_*} 
  && \thgrp[L\theta'+A]\bigparen{\cone_S^{[a',\infty]}(X)}
	  \arrow[two heads]{rd}{f_*}\arrow[shift left]{u}{p_*}
  && \thgrp[L\theta''+A]\bigparen{\cone_S^{[a'',\infty]}(X)}\arrow[]{u}{p_*} \\
  &\thgrp[\theta']\bigparen{\cone_T^{[b,\infty]}(Y)}
	  \arrow[two heads]{ru}{\bar f_*}  
  && \thgrp[\theta'']\bigparen{\cone_T^{[b',\infty]}(Y)} 
	  \arrow[two heads]{ur}{\bar f_*} \arrow[]{d}{p*}
  &	\\
  &\thgrp[\theta']\bigparen{\cone_T^{b}(Y)} \arrow[]{u}{\iota_*}\arrow[dashed]{rr}
  && \thgrp[\theta'']\bigparen{\cone_T^{b}(Y)}
  &
  \end{tikzcd}.
 \]
 To conclude, note that the dashed homomorphisms are induced by functions that are close to the identity and that $\iota_*$ and $p_*$ are isomorphisms. Then observe that Lemma~\ref{lem:qi.induce.maps.on.disc.fund.grp} implies that the maps $f_*$ are also injective and hence all the maps are isomorphisms. 
\end{proof}


\nocite{Hum17,Khu14}
\bibliographystyle{amsalpha}
\bibliography{MainBibliography}{}

\end{document}

%% file: StdCommands_Oct18.tex
\usepackage[utf8]{inputenc}
\usepackage[T1]{fontenc}				
\usepackage{amsmath,amssymb,amsthm,amsbsy}
\usepackage{mathtools} 					
\usepackage{microtype} 					
\usepackage{bbm}					
\usepackage{bm}						
\usepackage[shortcuts]{extdash} 			
\usepackage{color}
\usepackage{enumerate}
\usepackage{graphicx}    
\usepackage{xspace}					
\usepackage{subcaption}					
\usepackage[colorlinks=true,linkcolor=black, citecolor=black]{hyperref}  
\usepackage{mparhack}					
\usepackage{accents}
\usepackage[mode=image|tex]{standalone}

\usepackage{tikz}
  \usetikzlibrary{arrows,matrix,decorations.markings,patterns,cd}
  \tikzset{
    >=latex,
    bnode/.style={circle,fill=black,draw=black},
    middlearrow/.style={	
	  decoration={markings, mark= at position 0.55 with {\arrow{#1}}},
	  postaction={decorate}
	  },
    double arrow/.style args={#1 colored by #2 and #3}{ 
	  -stealth,line width=#1,#2, 
	  postaction={draw,-stealth,#3,line width=(#1)/3,
                shorten <=(#1)/3,shorten >=2*(#1)/3}, 
	  }
	  
  }
\tikzset{
        hatch distance/.store in=\hatchdistance,
        hatch distance=10pt,
        hatch thickness/.store in=\hatchthickness,
        hatch thickness=2pt
    }
    \makeatletter
    \pgfdeclarepatternformonly[\hatchdistance,\hatchthickness]{flexible hatch}
    {\pgfqpoint{0pt}{0pt}}
    {\pgfqpoint{\hatchdistance}{\hatchdistance}}
    {\pgfpoint{\hatchdistance-1pt}{\hatchdistance-1pt}}%
    {
        \pgfsetcolor{\tikz@pattern@color}
        \pgfsetlinewidth{\hatchthickness}
        \pgfpathmoveto{\pgfqpoint{0pt}{0pt}}
        \pgfpathlineto{\pgfqpoint{\hatchdistance}{\hatchdistance}}
        \pgfusepath{stroke}
    }
    
    \pgfdeclarepatternformonly[\hatchdistance,\hatchthickness]{flexible revhatch}
    {\pgfqpoint{0pt}{0pt}}
    {\pgfqpoint{\hatchdistance}{-\hatchdistance}}
    {\pgfpoint{\hatchdistance-1pt}{\hatchdistance-1pt}}%
    {
        \pgfsetcolor{\tikz@pattern@color}
        \pgfsetlinewidth{\hatchthickness}
        \pgfpathmoveto{\pgfqpoint{0pt}{0pt}}
        \pgfpathlineto{\pgfqpoint{\hatchdistance}{-\hatchdistance}}
        \pgfusepath{stroke}
    }

		\newcommand{\NN}{\mathbb{N}}	
			
		\newcommand{\RR}{\mathbb{R}}	
\renewcommand{\SS}{\mathbb{S}}		\newcommand{\TT}{\mathbb{T}}

		\newcommand{\ZZ}{\mathbb{Z}}

\newcommand{\CG}{\mathcal{G}}

\newcommand{\CO}{\mathcal{O}}			
			
		\newcommand{\CT}{\mathcal{T}}

\makeatletter						
\DeclareFontFamily{OMX}{MnSymbolE}{}	
\DeclareSymbolFont{MnLargeSymbols}{OMX}{MnSymbolE}{m}{n}
\SetSymbolFont{MnLargeSymbols}{bold}{OMX}{MnSymbolE}{b}{n}
\DeclareFontShape{OMX}{MnSymbolE}{m}{n}{
    <-6>  MnSymbolE5    <6-7>  MnSymbolE6   <7-8>  MnSymbolE7
   <8-9>  MnSymbolE8   <9-10> MnSymbolE9  <10-12> MnSymbolE10
  <12->   MnSymbolE12	}{}
\DeclareFontShape{OMX}{MnSymbolE}{b}{n}{
    <-6>  MnSymbolE-Bold5   <6-7>  MnSymbolE-Bold6   <7-8>  MnSymbolE-Bold7
   <8-9>  MnSymbolE-Bold8   <9-10> MnSymbolE-Bold9  <10-12> MnSymbolE-Bold10
  <12->   MnSymbolE-Bold12  }{}
\let\llangle\@undefined
\let\rrangle\@undefined
\DeclareMathDelimiter{\llangle}{\mathopen}%
                     {MnLargeSymbols}{'164}{MnLargeSymbols}{'164}
\DeclareMathDelimiter{\rrangle}{\mathclose}%
                     {MnLargeSymbols}{'171}{MnLargeSymbols}{'171}
\makeatother

\DeclarePairedDelimiter\abs{\lvert}{\rvert}		
\DeclarePairedDelimiter\norm{\lVert}{\rVert}		
\DeclarePairedDelimiter\angles{\langle}{\rangle}	
\DeclarePairedDelimiter\aangles{\llangle}{\rrangle}	
\DeclarePairedDelimiter\paren{(}{)}			
\DeclarePairedDelimiter\brackets{[}{]}			
\DeclarePairedDelimiter\braces{\{}{\}}			

\makeatletter
\let\oldabs\abs	
\def\abs{\@ifstar{\oldabs}{\oldabs*}}			
\let\oldnorm\norm
\def\norm{\@ifstar{\oldnorm}{\oldnorm*}}
\makeatother

	\newcommand{\bigparen}[1]{\paren[\big]{#1}}
	\newcommand{\Bigparen}[1]{\paren[\Big]{#1}}
\newcommand{\bigbrack}[1]{\brackets[\big]{#1}}
\newcommand{\Bigbrack}[1]{\brackets[\Big]{#1}}
\newcommand{\bigbrace}[1]{\braces[\big]{#1}}	
\newcommand{\Bigbrace}[1]{\braces[\Big]{#1}}	
\newcommand{\bigmid}{\mathrel{\big|}}			

\newcommand{\dist}[1][]{{d^{#1}}}
\newcommand{\wdist}[1][]{{\delta^{#1}_\Gamma}} 
 
\newcommand{\wSdist}[1][]{{\delta^{#1}_{\hspace{-0.15 ex} S}}} 
\newcommand{\wTdist}[1][]{{\delta^{#1}_{\hspace{-0.15 ex} T}}}

\newcommand{\jp}{jumping\=/path\xspace}
\newcommand{\jps}{jumping\=/paths\xspace}

\newcommand{\jto}[1][]{\overset{#1}{\mathbin\rightsquigarrow}}
\newcommand{\thhom}{\mathrel{\sim_\theta}}
\newcommand{\frhom}{\mathrel{\sim_{f}}}
\newcommand{\jpgrp}{J_{\hspace{-0.4 ex}S}\Pi_1}
\newcommand{\shpa}{{\rm Sh_{path}}}
\newcommand{\shjp}{{\rm Sh_{jump}}}
\makeatletter
\def\thgrp{\@ifnextchar[{\@thgrpwith}{\@thgrpwithout}} 
\def\@thgrpwith[#1]{\pi_{1,#1}}
\def\@thgrpwithout{\pi_{1,\theta}}
\def\thhom{\@ifnextchar[{\@thhomwith}{\@thhomwithout}} 
\def\@thhomwith[#1]{\mathrel{\sim_{#1}}}
\def\@thhomwithout{\mathrel{\sim_\theta}}
\makeatother

\newsavebox{\discrmapgraphic}  
\savebox{\discrmapgraphic}{
  \begin{tikzpicture}[scale=1, every node/.style={transform shape}]
      \path[clip] (-3.1 pt, -2.8 pt) rectangle (3.1 pt, 4 pt);
      \fill (0,0) circle (0.78 pt) (0,-0.6 pt) node[draw=black] {$\widehat{}$};
      \end{tikzpicture}
  }
\newcommand{\discrmap}[1][]{
  \usebox{\discrmapgraphic}
  \raisebox{0.38 ex}{$\scriptscriptstyle #1$}
  }

\DeclareMathOperator{\id}{id}				
\DeclareMathOperator{\vol}{Vol}				

\DeclareMathOperator{\fix}{Fix}				

\DeclareMathOperator{\isom}{Isom}			

\DeclareMathOperator{\aut}{Aut}				

\DeclareMathOperator{\SO}{SO}

\DeclareMathOperator{\Sl}{SL}	

\DeclareMathAlphabet{\mathbit}{OT1}{cmr}{bx}{it}  	

\theoremstyle{plain}
\newtheorem{thm}{Theorem}[section]				
\newtheorem{prop}[thm]{Proposition}		
\newtheorem{lem}[thm]{Lemma}						
\newtheorem{cor}[thm]{Corollary}

\newtheorem*{thm*}{Theorem}			\newtheorem*{theorem*}{Theorem}		
\newtheorem*{prop*}{Proposition}		\newtheorem*{proposition*}{Proposition}
\newtheorem*{lem*}{Lemma}			\newtheorem*{lemma*}{Lemma}			
\newtheorem*{cor*}{Corollary}			\newtheorem*{corollary*}{Corollary}
\newtheorem*{qu*}{Question}			\newtheorem*{question*}{Question}
\newtheorem*{conj*}{Conjecture}			\newtheorem*{conjecture*}{Question}
\newtheorem*{fact*}{Fact}
\newtheorem*{claim*}{Claim}

\theoremstyle{definition}
\newtheorem{de}[thm]{Definition}

\newtheorem*{de*}{Definition}			\newtheorem{definition*}{Definition}	
\newtheorem*{notation*}{Notation}	
\newtheorem*{conv*}{Convention}			\newtheorem*{convention*}{Convention}

\theoremstyle{remark}
\newtheorem{rmk}[thm]{Remark}						
\newtheorem{exmp}[thm]{Example}
